\documentclass{amsart}
\usepackage{graphicx}

\usepackage{amsmath, mathrsfs, amssymb, eucal, amsthm, amscd, amsrefs, amsxtra}

\usepackage{placeins}
\usepackage{microtype}
\usepackage{hyperref}

\hypersetup{
colorlinks=true,
linkcolor=blue,
citecolor=red,
filecolor=magenta,      
urlcolor=cyan,
pdftitle={Mordell_Curves_Arithmetic_Progression},
pdfpagemode=FullScreen,
}

\usepackage{amsrefs}

\usepackage{cleveref}

\numberwithin{equation}{section}
\newtheorem{theorem}{\bf Theorem}[section]
\newtheorem{lemma}[theorem]{\bf Lemma}
\newtheorem{proposition}[theorem]{\bf Proposition}

\theoremstyle{remark}

\begin{document}

\title[On certain $D(9)$ and $D(64)$ Diophantine triples]{On certain $D(9)$ and $D(64)$ Diophantine triples}

\author[B. Earp-Lynch]{Benjamin Earp-Lynch}
\address{Department of Mathematics, Brock University, Ontario, Canada L2S 3A1}
\email{be09wg@brocku.ca}

\author[S. Earp-Lynch]{Simon Earp-Lynch}
\address{Department of Mathematics, Brock University, Ontario, Canada L2S 3A1}
\email{se09ae@brocku.ca}

\author[O. Kihel]{Omar Kihel}
\address{Department of Mathematics, Brock University, Ontario, Canada L2S 3A1}
\email{okihel@brocku.ca}

\keywords{Diophantine triple, Pell's equation, Fibonacci numbers, Linear forms in logarithms}

\subjclass[2010]{11D09, 11D45, 11B37, 11J86}

\begin{abstract} A set of $m$ distinct positive integers $\{a_{1},\dots a_{m}\}$ is called a $D(q)$-$m$-tuple for nonzero integer $q$ if the product of any two increased by $q$, $a_{i}a_{j}+q$, $i\neq j$ is a perfect square. Due to certain properties of the sequence, there are many $D(q)$-Diophantine triples related to the Fibonacci numbers. A result of Ba\'{c}i\'{c} and Filipin characterizes the solutions of Pellian equations that correspond to $D(4)$-Diophantine triples of a certain form. We generalize this result in order to characterize the solutions of Pellian equations that correspond to $D(l^2)$-Diophantine triples satisfying particular divisibility conditions. %
Subsequently, we employ this result and bounds on linear forms in logarithms of algebraic numbers in order to classify all $D(9)$ and $D(64)$-Diophantine triples of the form $\{F_{2n+8},9F_{2n+4},F_{k}\}$ and $\{F_{2n+12},16F_{2n+6},F_{k}\}$, where $F_{i}$ denotes the $i$th Fibonacci number.\end{abstract}

\maketitle

\section{Introduction} \label{introduction}

A set of $m$ (integer) elements $\{a_1,\dots,a_{m}\}$ is called an (integer) Diophantine $m$-tuple with property $D(l)$, or alternatively a $D(l)$-$m$-tuple, provided the product of any two different elements increased by $l$, $a_{i}a_{j}+l$ where $i\neq j$, is an integer square. Similar sets have been studied for centuries. Fermat showed that the set $\{1,3,8\}$ can be extended from a $D(1)$-triple to a $D(1)$-quadruple with the additional element $120$. Baker and Davenport \cite{B-D} showed that $120$ is in fact the only integer that extends this set from a $D(1)$-triple to a $D(1)$-quadruple. This was also the first appearance of the method that came to be known as Baker-Davenport Reduction. He, Togb\'{e} and Ziegler \cite{Z-al} proved that there does not in fact exist a $D(1)$-quintuple. $D(l)$-$m$-$tuples$ where $l=-1$ and $l=4$ have also been studied (see \cite{Kihel}, \cite{B-F}, \cite{F-al} and \cite{He-al} and \cite{T-al}) as have further cases, though all less extensively than the case when $l=1$.

The $D(1)$-triple $\{1,3,8\}$ is an example of a $D(1)$-triple of Fibonacci numbers. Let $F_{n}$ denote the $n$th Fibonacci number, and let $F_{1}=F_{2}=1$. Using the recurrence relation, $F_{n+1}=F_{n}+F_{n-1}$, the sequence of Fibonacci numbers is seen to be as follows, $$1,1,2,3,5,8,13,21,34,55,\dots.$$ The Fibonacci numbers obey what is called Binet's formula, $$F_{n}=\frac{\alpha^{n}-\overline{\alpha}^{n}}{\sqrt{5}}$$ where here $\alpha=\frac{1+\sqrt{5}}{2}$ and $\overline{\alpha}$ is its conjugate, $\overline{\alpha}=\frac{1-\sqrt{5}}{2}$. Note that the $D(1)$-triple $\{1,3,8\}$ consists of the first $3$ even-indexed Fibonacci numbers. Using Catalan's identity,
\[
F_{n}^2-F_{n-r}F_{n+r}=(-1)^{n-r}F_{r}^{2},
\]
it can be determined that in fact any set of three consecutive even-indexed Fibonacci numbers, $\{F_{2n},F_{2n+2},F_{2n+4}\}$ will be a $D(1)$-triple. Generalizing this form of $D(1)$-triple to $D(l)$-triples with $l$ an integer not necessarily equal to $1$ requires the sequence of Lucas numbers, whose $n$th element is denoted $L_{n}$. Define $L_{1}=1$, $L_{2}=3$ and $L_{n+1}=L_{n}+L_{n-1}$. Using Catalan's identity and the factorization of even-indexed Fibonacci numbers, $F_{2n}=L_{n}F_{n}$, we can form the family of $D(F_{2r}^{2})$-triples, $$\{F_{2n},L_{r}F_{2n+2r},F_{2n+4r}\}.$$ The $r=1,2,3$ cases of corresponding quadruples are seen in \cite{Duj}. In \cite{T-al}, it was shown that if $\{F_{2n+6},4F_{2n+4},F_{k}\}$ is a $D(4)$-triple, then $k=2n$ unless $n=1$ in which case $k=1$ or $k=2$, due to the fact that $F_{1}=F_{2}=1$.

In this paper, we prove the following two theorems.
\begin{theorem}\label{first theorem}
If $\{F_{2n+8},9F_{2n+4},F_{k}\}$ is a $D(9)$-triple, then if $n>1$, we must have $k=2n$. If $n=1$, we have $F_{1}=1=F_{2}$, so $k=1$ and $k=2$ are both solutions.
\end{theorem}
\begin{theorem}\label{second theorem}
 If $\{16F_{2n+6},F_{2n+12},F_{k}\}$ is a $D(64)$-triple and $3\mid n$, then we must have $k=2n$.
\end{theorem}
The proofs of these theorems rely on a lemma on Pellian equations, which will be proven in Section \ref{PellLemma}, and on bounds for linear forms in logarithms of algebraic numbers.

\section{Preliminaries}\label{Preliminaries}

Here we state several results that will be required in order to prove  \Cref{first theorem,second theorem}.

The following result is a variation of the original lemma of Baker and Davenport \cite{B-D} due to Dujella and Peth\H{o} \cite{D-P}. 
\begin{lemma}[Variation on a Lemma of Baker and Davenport]
Assume that $\kappa$ and $\mu$ are real numbers and $M$ is a positive integer. Let $P/Q$ be the convergent of the continued fraction expansion of $\kappa$ such that $Q>6M$ and let
\[
\eta =||\mu Q||-M\cdot ||\kappa Q||,
\]
where $||\cdot||$ denotes the distance from the nearest integer.  If $\eta>0$, then there is no solution of the inequality
\[
0<j\kappa-k+\mu<AB^{-j}
\]
in integers $j$ and $k$ with
\[
\frac{\log{(AQ/\eta)}}{\log{B}}\leq j\leq M.
\]
\end{lemma}

Recall that if $[a_{0};a_1,a_2,\dots]$ is the simple regular continued fraction representation of some real number, the integers $a_i$ are called its partial quotients and the rational approximation $[a_{0};a_{1},a_{2},\dots,a_{n}]$ is called its $n$th convergent. The following result (see the proof of Theorem 22 in \cite{Kin}) gives bounds for how well a convergent approximates an irrational number in terms of the subsequent partial quotient in its continued fraction representation. It will be listed as a lemma here so that it can be easily called upon later in the paper.
\begin{lemma}\label{ParQuo}
Let $\theta$ be an irrational number, $\dfrac{p_{r}}{q_{r}}$ its $r$th convergent and $a_{r+1}$ its $r+1$st partial quotient. The following inequality holds:
\[
\frac{1}{q_{r}^{2}(a_{r+1}+2)}<\Big|\theta-\frac{p_{r}}{q_{r}}\Big|\leq\frac{1}{q_{r}^{2}a_{r+1}}.
\]
\end{lemma}

The following well-known theorem due to Legendre establishes a condition which can determine whether a given rational approximation of an irrational number is a convergent of the irrational's continued fraction (see \cite{Bu-F}). We list it so that it can be referred to later in the paper.

\begin{theorem}[Legendre]\label{Legendre}
Let $\theta$ be an irrational number. Let $p,q$ be integers such that $q\geq 1$ and
\[
\Big|\theta-\frac{p}{q}\Big|<\frac{1}{2q^{2}}.
\]
Then $\frac{p}{q}$ is a convergent of $\theta$.
\end{theorem}

For any non-zero algebraic number $\gamma$ of degree $d$ over $\mathbb{Q}$ whose minimal polynomial over $\mathbb{Z}$ is $a\Pi_{j=1}^{d}(X-\gamma^{(j)})$, we denote by
\[
h(\gamma)=\frac{1}{d}\Big( \log{a}+\sum_{j=1}^{d}\log{}\max{(1,|\gamma^{(j)}|)}  \Big)
\]
its absolute logarithmic height. Here the $\gamma^{(j)}$ run through all the solutions to the minimal polynomial of $\gamma$, i.e. $\gamma$ and all its conjugates.

The following two lemmas are identical to the first two listed in Section 2 of \cite{T-al}. The first, due to Matveev, establishes a bound on a linear form in three logarithms. The second, due to Laurent, establishes a bound on a linear form in two logarithms. See \cite{Laur}, and \cite{Mat}. We use both in order to obtain usable bounds on all relevant parameters.

\begin{lemma}\label{Matveev}
Let $\Lambda$ be a linear form in logarithms of multiplicatively independent totally real algebraic numbers $\alpha_{1},\dots,\alpha_{N}$ with rational integer coefficients $b_{1},\dots,b_{N}$ ($b\neq 0$). Let $h(\alpha_{j})$ denote the absolute logarithmic height of $\alpha_{j}$ for $1\leq j\leq N$. Define the numbers $D,A_{j}$ $(1\leq j\leq N)$ and $E$ by 
\begin{align*}
D:=&[\mathbb{Q}(\alpha_{1},\dots,\alpha_{N}) : \mathbb{Q}],\\
A_{j}:=&\max{\{Dh(\alpha_{j}),|\log{\alpha_{j}}|\}},\\
E:=&\max{\{ 1,\max{\{|b_{j}|A_{j}/A_{N};1\leq j\leq N\}}\}}.
\end{align*}
Then
\[
\log{|\Lambda|}>-C(N)C_{0}W_{0}D^{2}\Omega,
\]
where
\begin{align*}
C(N)&:=\frac{8}{(N-1)!}(N+2)(2N+3)(4e(N+1))^{N+1},\\
C_{0}&:=\log{(e^{4.4N+7}N^{5.5}D^{2}\log{(eD)})},\\
W_{0}&:=\log{(1.5eED\log{(eD)})},\text{     } \Omega=A_{1}\dots A_{N}.
\end{align*}
\end{lemma}

\begin{lemma}\label{Laurent}
Let $\gamma_{1}>1$ and $\gamma_{2}>1$ be two real multiplicatively independent algebraic numbers, $b_{1},b_{2}\in\mathbb{Z}$ not both $0$ and
\[
\Lambda=b_{2}\log{\gamma_{2}}-b_{1}\log{\gamma_{1}}.
\]
Let $D:=[\mathbb{Q}(\gamma_{1},\gamma_{2}):\mathbb{Q}]$. Let
\[
h_{i}\geq\max{\Big\{ h(\gamma_{i}),\frac{|\log{\gamma_{i}}|}{D},\frac{1}{D}  \Big\}}\text{ } for\text{ }i=1,2,\text{     } b'\geq\frac{|
b_{1}|}{Dh_{2}}+\frac{|b_{2}|}{Dh_{1}}.
\]
Then
\[
\log{|\Lambda |}\geq -17.9\cdot D^4\Big(\max{\Big\{\log{b'}+0.38,\frac{30}{D},1   \Big\}}\Big)^2h_{1}h_{2}.
\]
\end{lemma}
%


The following two lemmas will be useful in applying \Cref{Matveev}.

\begin{lemma}\label{first nosquare}
$F_{2n+4}F_{2n+8}$ is neither a square nor $5$ times a square.
\end{lemma}

\begin{proof}
To see that $F_{2n+4}F_{2n+8}$ is not a square, simply note that $F_{2n+4}F_{2n+8}+1=F_{2n+6}^{2}$. Suppose that $F_{2n+4}F_{2n+8}$ is $5$ times a square, so $F_{2n+4}F_{2n+8}=5y^{2}$ for some integer $y$. Then we must have that $5y^{2}+1=F_{2n+6}^{2}$. We examine the solutions of the Pellian equation
\begin{equation}
X^{2}-5Y^{2}=1,
\end{equation}
with our interest lying in those solutions $X+Y\sqrt{5}$ for which $X=F_{2n+6}$ for some positive integer $n$.

This Pellian equation has fundamental solution $9+4\sqrt{5}=\alpha^{6}$, where here $\alpha=\dfrac{1+\sqrt{5}}{2}$. Hence all solutions to equation $(1)$ have the form $(9+4\sqrt{5})^{j}$ for some $j\in\mathbb{Z}^{+}$. If we let $X_{j}+Y_{j}\sqrt{5}=(9+4\sqrt{5})^{j}$, then $X_{j}=\dfrac{(9+4\sqrt{5})^{j}+(9+4\sqrt{5})^{-j}}{2}=\dfrac{\alpha^{6j}+\alpha^{-6j}}{2}$. Since $F_{k}=\dfrac{\alpha^{k}-(\frac{-1}{\alpha})^{k}}{\sqrt{5}}$, our aim becomes to solve the equation
\begin{equation}
\frac{\alpha^{6j}+\alpha^{-6j}}{2}=\frac{\alpha^{2n+6}-\alpha^{-2n-6}}{\sqrt{5}}
\end{equation}
for $n,j\in\mathbb{Z}^{+}$.

\vspace{5mm}

We see here that for $n\geq 1$,
\begin{gather*}
\frac{\alpha^{2n+4}+\alpha^{-2n-4}}{2}<\frac{\alpha^{4}}{2}\alpha^{2n}+\frac{1}{2}<8\alpha^{2n}-1\\
<\frac{\alpha^{6}\alpha^{2n}-\alpha^{-2n-6}}{\sqrt{5}}=\frac{\alpha^{2n+6}-\alpha^{2n-6}}{\sqrt{5}},
\end{gather*}
hence we must have $j>\frac{n+2}{3}$. However,
\[
\frac{\alpha^{2n+6}+\alpha^{-2n-6}}{2}>\frac{\alpha^{2n+6}+\alpha^{-2n-6}}{\sqrt{5}}>\frac{\alpha^{2n+6}-\alpha^{-2n-6}}{\sqrt{5}},
\]
which means $j<\frac{n+3}{3}$.

\vspace{2mm}

The bounds $\frac{n+2}{3}<j<\frac{n+3}{3}$ mean that $j$ cannot be an integer, which is a contradiction. Thus there is no solution to the Pellian equation $(1)$ wherein $X$ is a Fibonacci number with index $2n+6,$  $n\in\mathbb{Z}^{+}$, which means that $F_{2n+4}F_{2n+8}$ is neither a square nor $5$ times a square.
\end{proof}

\begin{lemma}\label{second nosquare}
$F_{2n+6}F_{2n+12}$ is neither a square nor $5$ times a square.
\end{lemma}
\begin{proof}

To see that it is not itself a square, observe that $F_{2n+6}F_{2n+12}+4=F_{2n+9}^{2}$, and the difference between any two nonzero, nonidentical squares is either odd or bigger than $4$. Suppose that $F_{2n+6}F_{2n+12}$ is $5$ times a square, so $F_{2n+6}F_{2n+12}=5y^{2}$ for some integer $y$. We're interested in solutions to the Pellian equation
\begin{equation}
X^{2}-5Y^{2}=4
\end{equation}
with $X=F_{2n+9}$ for some $n\in\mathbb{Z}^{+}$. Since the fundamental solution of $X^{2}-5Y^{2}=1$ is $9+4\sqrt{5}$, a theorem of Nagell (see \cite{Nag}) gives us that all classes of solutions to equation $(3)$ have fundamental solution $u+v\sqrt{5}$ with

\begin{gather*}
0<v\leq \frac{4}{\sqrt{2(9+1)}}\sqrt{4}=\frac{4}{\sqrt{5}}<2 \text{ and}\\
0\leq |u|\leq\sqrt{\frac{1}{2}(9+1)\cdot 4}=2\sqrt{5}<5,
\end{gather*}
so all solutions to $(3)$ are of the form
\[
(\pm 3+\sqrt{5})(9+4\sqrt{5})^{j}
\]
for some integer $j\geq 1$. If we let $V_{j}+U_{j}\sqrt{5}=(9+4\sqrt{5})^{j}$, then we have solutions $X_{j}+Y_{j}\sqrt{5}$, where $X_{j}=\pm 3V_{j}+5U_{j}$ and $Y_{j}=V_{j}\pm 3U_{j}$. This means that we want to find solutions $(j,n)$ to the equation
\[
X_{j}=F_{2n+9}=\frac{\alpha^{2n+9}-(\frac{-1}{\alpha})^{2n+9}}{\sqrt{5}}=\frac{\alpha^{2n+9}+\alpha^{-2n-9}}{\sqrt{5}}.
\]
Since $(9+4\sqrt{5})=\alpha^{6}$, we have $V_{j}=\frac{\alpha^{6j}+\alpha^{-6j}}{2}$ and our aim is to solve the following
\[
\pm 3(\frac{\alpha^{6j}+\alpha^{-6j}}{2})+5(\frac{\alpha^{6j}-\alpha^{-6j}}{2\sqrt{5}})=\frac{\alpha^{2n+9}+\alpha^{-2n-9}}{\sqrt{5}}.
\]
Noting that $-3\sqrt{5}+5$ and $-3\sqrt{5}-5$ are both less than $0$ and will never yield a solution, we obtain after cancelling denominators:
\begin{equation}
(3\sqrt{5}+5)\alpha^{6j}+(3\sqrt{5}-5)\alpha^{6j}=2\alpha^{2n+9}+2\alpha^{-2n-9}.
\end{equation}

We see from a brief observation that
\begin{gather*}
(3\sqrt{5}+5)\alpha^{2n+6}+(3\sqrt{5}-5)\alpha^{-2n-6}>2.76\alpha^{2n+9}+2.76\alpha^{-2n-9}\\
>2\alpha^{2n+9}+2\alpha^{-2n-9},
\end{gather*}
meaning that $j<\frac{n+3}{3}$, and
\begin{gather*}
(3\sqrt{5}+5)\alpha^{2n+4}+(3\sqrt{5}-5)\alpha^{-2n-4}<1.06\alpha^{2n+9}+18.95\alpha^{2n-9}\\
<2\alpha^{2n+9}<2\alpha^{2n+9}+2\alpha^{-2n-9}
\end{gather*}
which gives us $\frac{n+2}{3}<j<\frac{n+3}{3}$, contradicting $j\in\mathbb{Z}$. Thus we see that $F_{2n+6}F_{2n+12}$ is neither a square nor $5$ times a square.
\end{proof}


\section{Lemma on Pellian Equations}\label{PellLemma}

In this section, we generalize a lemma of Filipin and Ba\'ci\'c (Lemma 1 in \cite{B-F}), which is similar to a result in \cite{J} and was subsequently used in \cite{T-al}. We use the same ideas as in the proof in \cite{B-F}, however there was a mistake in the proof which is amended here.

\vspace{5mm}

Let $l$ be a positive integer and $\{a,b,c\}$ be a $D(l^{2})$-triple, i.e. there exist positive integers $r,s,t$ such that
\[
ab+l^2=r^2,\quad ac+l^2=s^2, \quad\text{and}\quad bc+l^2=t^2.
\]
\begin{lemma}\label{Filipin}
Let $\{a,b,c\}$ be a $D(l^2)$ triple with $a<b<a\Big(4+\dfrac{4}{l^2}\Big)$ and assume that one of the following conditions holds:
\renewcommand{\labelenumi}{\roman{enumi})}
\begin{enumerate}
\item $l=2$,
\item $l$ is an odd prime and $l\mid ab$, or
\item $l^2\mid a$ or $l^2\mid b$.
\end{enumerate}
Then all solutions of the equation
\begin{equation} \label{eq:Pellian Equation 1}
at^2-bs^2=l^2(a-b)
\end{equation}
are of the form
\[
t\sqrt{a}+s\sqrt{b}=(\pm l\sqrt{a}+l\sqrt{b})\Big(\frac{r+\sqrt{ab}}{l}\Big)^{\nu},
\]
where $\nu\in\mathbb{Z}^{+}$.
\end{lemma}
\begin{proof}
Define $s'=\dfrac{rs-at}{l}$, $t'=\dfrac{rt-bs}{l}$ and $c'=\dfrac{(s')^2-l^2}{a}$. The triple $\{a,b,c'\}$ is also a $D(l^2)$-triple. In the case of $l=2$ (the lemma found in \cite{B-F}), $2\mid(rs-at)$ and $2\mid(rt-bs)$ no matter whether $a$ and $b$ are both odd or not. For $l>2$ this may not be true. If $l$ is prime, $l\mid ab$ and if not, $l^2\mid a$ or $l^2\mid b$, so we must have that $l\mid r$, and in addition that $l\mid s$ or $l\mid t$, thus $s'$ and $t'$ are always integers. Moreover, since
\[
t\sqrt{a}+s\sqrt{b}=(t'\sqrt{a}+s'\sqrt{b})\Big(\frac{r+\sqrt{ab}}{l}\Big),
\]
$(t',s')$ belongs to the \emph{same class} of solutions of \eqref{eq:Pellian Equation 1} as $(t,s)$, thus we can replace $c=c'$ and follow the process again with the triple $\{a,b,c'\}$ while always remaining in the same class of solutions. This will be the key to our proof. The above information means that if we let
\[
t_{\nu}^{\pm}\sqrt{a}+s_{\nu}^{\pm}\sqrt{b}=(\pm l\sqrt{a}+l\sqrt{b})\Big(\frac{r+\sqrt{ab}}{l}\Big)^{\nu},
\]
then if after a certain number of times repeating the process of finding $c'$ and replacing $c=c'$ in our triple, we have that $(t',s')=(\pm l,l)$, or equivalently, if $c'=a+b\pm2r$, then we must have that $(t,s)=(t_{\nu}^{\pm},s_{\nu}^{\pm})$ for some positive integer $\nu$. Thus at this point the proof will be complete.

\vspace{5mm}
\begin{itemize}
\item $s'$ is always positive.
\noindent To see this, observe the Pellian equation:
\[
bs^2-at^2=l^2(b-a)
\]
obtained by multiplying \eqref{eq:Pellian Equation 1} by $-1$. Multiplying by $a$, we get
\[
abs^2-a^2t^2=(r^2-l^2)s^2-a^2t^2=r^2s^2-a^2t^2-l^2s^2=l^2a(b-a)
\]
which means that
\[
(rs-at)(rs+at)=l^2a(b-a)+l^2s^2.
\]
Since $r,s,a,t,l$ are all positive and $0<a<b$, we must have that $rs-at=ls'>0$.

\vspace{5mm}
\item $c'\geq 0$.
If $l$ is prime and $l\mid a$, then $l\mid ab+l^2=r^2$, which means $l\mid r$ and similarly, $l\mid s$. If $l$ also divides $b$, then $l\mid t$, and so $l^2\mid rs-at$, meaning $s'\geq l$, and so $c'\geq 0$. If $l\nmid b$, then $ab+l^2=r^2$ implies $l^2\mid a$, giving $l^2\mid rs-at$, and again $c'\geq 0$. A similar process proves the remark in the case $l$ prime and $l\mid b$.

If $l$ is composite and $l^2\mid a$, then $l^2\mid ab+l^2=r^2$ and $l^2\mid ac+l^2=s^2$. So $l^2\mid rs-at$, which means that
\[
l\leq \frac{rs-at}{l}=s',
\]
and $c'=\dfrac{(s')^2-l^2}{a}\geq 0$.

Similarly if $l^2\mid b$, then $l^2\mid ab+l^2=r^2$ and $l^2\mid bc+l^2=t^2$. So $l^2\mid rt-bs$, which means that
\[
l\leq \Big|\frac{rt-bs}{l}\Big|=|t'|.
\]
Therefore $c'=\dfrac{(t')^2-l^2}{b}\geq 0$.

If $l=2$, we show that $c'\geq 0$ regardless of whether it divides $a$ or $b$. If $2\mid ab$, then the proof is the same as above since we have $l$ prime and $l\mid ab$. Suppose that both $a$ and $b$ are odd and suppose for a contradiction that $c'<0$. This means that $s'^2-4<0$, which is equivalent to $s'<2$. Since $s'$ is always positive, this must mean that $s'=1$, so we set $\dfrac{rs-at}{2}=1$. Multiplying by 2 and adding $at$ to each side, then squaring, we get $r^2s^2=(2+at)^2$, giving $a^2bc+4ab+4ac+16=4+4at+a^2bc+4a^2$. We subtract $a^2bc$, divide by $4$ and rearrange to obtain
\[
-3=a(b+c-a-t).
\]
Since we assumed that $a$ and $b$ are both $1\pmod{2}$, this must mean that $c-t\equiv 1\pmod{2}$. However because of the fact that $b\equiv 1\pmod{2}$, we must have that $t^2= bc+4\equiv c\pmod{2}$, and so $t\equiv c\pmod{2}$. Therefore $a(b+c-a-t)\equiv 0\pmod 2$ and we have obtained a contradiction. Hence $c'\geq 0$ when $l=2$.

\item{Case when $0\leq c'<b$.}
If $c'=0$, then $s'=t'=l$.  So $c=a+b+c'+\frac{2}{l^2}(abc'+rs't')=a+b+2r$, and the proof is complete.

\vspace{5mm}

Suppose that $0<c'<b$. Define $r'=(s')'=\dfrac{rs'-at'}{l}$ and $b'=(c')'=\dfrac{(r')^2-l^2}{a}$. Then $b'=(c')'=a+b+c'+\dfrac{2}{l^2}(abc'-rs't')$. We have that $\{a,(c')',c',b\}$ is a regular $D(l^2)$-quadruple. Indeed
\begin{align*}
ab+l^2=r^2,&&\quad ac'+l^2=(s')^2,&&\quad bc'+l^2=(t')^2\\
ab'+l^2=(r')^2,&&\quad bb'+l^2=(q')^2,&&\quad b'c'+l^2=u^2
\end{align*}
where
\[
u=\frac{s't'-rc'}{l}\quad\text{and}\quad q'=(t')'=\frac{rt'-bs'}{l}\text{ both}\in\mathbb{Z}.
\]
Because $\{a,c',b\}$ is a $D(l^2)$-triple, we have
\[
ac'+l^2=(s')^2,\quad
ab+l^2=r^2,\quad
bc'+l^2=q^2,
\]
and
\[
r=\frac{(s')r'+aq'}{l},\quad
z=\frac{(s')q'+cr'}{l}
\]
\[
r'=\frac{(s')r-aq}{l},\quad
q'=\frac{(s')q-c'r}{l}
\]
It can be seen that
\begin{align*}
ab+l^2=r^2&=\Big(\frac{(s')r'+aq'}{l}\Big)^2=\frac{1}{l^2}((s')^2(ab'+l^2)+a^2(b'c'+l^2)+2(s')ar'q')\\
&=\frac{1}{l^2}((s')^2ab'+l^2ac'+l^4+a^2(b'c'+l^2)+2(s')r'aq'),
\end{align*}
so
\begin{align*}
b&=\frac{1}{l^2}(s')^2b'+c'+\frac{1}{l^2}ab'c'+a+\frac{2}{l^2}(s')r'q'\\
&>\frac{2}{l^2}ab'c'+c'+a+\frac{2}{l^2}\sqrt{ac'}\sqrt{ab'}\sqrt{b'c'}\\
&=\frac{4}{l^2}ab'c'.
\end{align*}
Hence 
\[
b'<\dfrac{l^2b}{4ac'}<\dfrac{\dfrac{l^2}{4}(4+\dfrac{4}{l^2})}{ac'}=\dfrac{l^2+1}{c'},
\]
 which means $b'c'<l^2+1$. But since $\{a,b',c'\}$ is a $D(l^2)$-triple,  $(c')'=b'=0$, so $c'=a+b-2r$.

\vspace{5mm}
\item{We may assume $c'<b$.}
The proof proceeds by showing that we can assume $c'<b$. It is here that the mistake in \cite{B-F} arose. The proof of Lemma 1 in that paper used a version of inequalities (3.7) and (3.8) found in \cite{J}, modified for $D(4)$-triples in order to show that it could be assumed that $c'<r^{2}$. The inequality was:
\[
ac'=(s')^2-4<\frac{s^2}{r^2}-4=\frac{ac+4}{r^2}-4<\frac{ac}{r^2}.
\]
In \cite{J}, the above inequality was shown to hold for $D(1)$-triples provided we assume that $b\leq a+c$, but in the case of $D(4)$-triples this inequality will never hold. In particular,
\begin{equation}\label{eq:Mistake}
s'<\frac{s}{r}
\end{equation}
is never true. To see this, note that we have
\begin{align*}
4s&=s(r^2-ab)=r^2s-art+art+abs\\
&=r(rs-at)+a(rt-bs)\\
&=\frac{4(rs'+at')}{2}.
\end{align*}
So $s=\dfrac{rs'+at'}{2}$, which means inequality \eqref{eq:Mistake} is
\[
rs'<\frac{rs'+at'}{2},
\]
which is equivalent to
\[
rs'<at'.
\]
Since $s'$ is always positive the inequality will always fail when $t'\leq 0$, so we assume that $t'>0$. Squaring both sides, we obtain the equivalent inequality
\begin{equation}\label{eq:Mistake 2}
r^2(s')^2<a^2(t')^2.
\end{equation}
Substituting $r^2=ab+4,(s')^2=ac'+4,(t')^2=bc'+4$ into \eqref{eq:Mistake 2}, we get
\begin{align*}
r^2(s')^2&=
a^2bc'+4ab+4ac'+16<a^2(t')^2
=a^2bc'+4a^2
\end{align*}
so our original inequality is equivalent to
\[
ab+ac'+4<a^2
\]
which is never true since $a<b$.

\vspace{5mm}

Now let $\{a,b,c\}$ be a $D(l^2)$-triple with $r,s,t,c',s',t'$ defined as above. If $c<b$ then we are done, since $c=\overline{c}'$ for the triple $\Big\{a,b,\overline{c}=a+b+c+\dfrac{2}{l^2}(abc+rst)\Big\}$ which falls under the previous case (since $0\leq\overline{c}'<b$).

Suppose that $c\geq b>a>0$. We want to show that we can assume $c'<b$. To do this, we suppose that $c'\geq c$, and show that a contradiction arises. This means that if $c\geq b$ then we must have $c'<c$, meaning that we can keep on replacing $c'=c$ and repeating with the new triple and eventually we will obtain $c'<b$.

If $c'\geq c$, then
$c'\geq a+b+c'+\frac{2}{l^2}(abc'+rs't')$, meaning
\[
a+b+\frac{2}{l^2}(abc'+rs't')\leq 0
\]
which implies $t'<0$. So we must have $rt-bs<0$, and hence $r^2t^2<b^2s^2$, which gives $(ab+l^2)(bc+l^2)<b^2(ac+l^2)$.
Rearranging this last expression gives $$ab+bc+l^2<b^2,$$ which means $c<b-a$, contradicting our assumption that $c\geq b$.
\end{itemize}

\vspace{5mm}

Therefore any $D(l^2)$-triple $\{a,b,c\}$ with $a<b<4(1+\dfrac{4}{l^2})$ can, through repetition of this process of taking $c'=c$, be reduced to a triple for which $0\leq c'<b$, and we have established that the Pellian equation \eqref{eq:Pellian Equation 1} stemming from such triples has its solutions in the class of $(\pm l,l)$.
\end{proof}



\section{The $D(9)$-Triple}\label{D(9)}

The next two sections follow a procedure which resembles that contained in \cite{T-al}, where a similar result on $D(4)$ triples of the form $\{F_{k}, 4F_{2n+4}, F_{2n+6}\}$ was proven.

\vspace{5mm}

Let us begin by setting up a Pellian equation. Given that $\{F_{2n+8},9F_{2n+4},F_{k}\}$ is a $D(9)$-triple, we must have that
\begin{equation}\label{eq:XY}
F_{2n+8}F_{k}+9=X^{2},\text{ and } 9F_{2n+4}F_{k}+9=Y^{2}
\end{equation}
for some integers $X$ and $Y$.
From this, we obtain
\begin{equation}
F_{2n+8}Y^{2}-9F_{2n+4}X^{2}=9(F_{2n+8}-9F_{2n+4}).
\end{equation}
Since $3\mid 9F_{2n+4}$, $F_{2n+8}<9F_{2n+4}<\frac{40}{9}F_{2n+8}$, and $9F_{2n+4}F_{2n+8}+9=(3F_{2n+6})^{2}$, Lemma \ref{Filipin} tells us that all solutions $Y\sqrt{F_{2n+8}}+X\sqrt{9F_{2n+4}}$ of this equation are given by
\[
Y\sqrt{F_{2n+8}}+3X\sqrt{F_{2n+4}}=(\pm 3\sqrt{F_{2n+8}}+9\sqrt{F_{2n+4}})(F_{2n+6}+\sqrt{F_{2n+4}F_{2n+8}})^{j}.
\]
Define the sequences $\{V_{j}\}_{j=1}^{\infty}$ and $\{U_{j}\}_{j=1}^{\infty}$ by 
\[
V_{j}+U_{j}\sqrt{F_{2n+4}F_{2n+8}}:=(F_{2n+6}+\sqrt{F_{2n+4}F_{2n+8}})^{j}.
\]
This gives
\begin{align*}
X&=X_{j}=3V_{j}\pm F_{2n+8}U_{j},\text{ and   }\\
Y&=Y_{j}=\pm V_{j}+9F_{2n+4}U_{j}.
\end{align*}
Substituting these expressions into equations \eqref{eq:XY}, we obtain
\begin{align*}
F_{2n+8}F_{k}+9&=X^{2}=(3V_{j}\pm F_{2n+8}U_{j})^{2}\text{, and}\\
9F_{2n+4}F_{k}+9&=Y^{2}=(\pm 3V_{j}+9F_{2n+4}U_{j})^{2}.
\end{align*}
This gives alternative expressions for $F_{k}$
\[
F_{k}=\frac{9V_{j}^{2}-9}{F_{2n+8}}+F_{2n+8}U_{j}^{2}\pm 6U_{j}V_{j} \text{ and    } F_{k}=\frac{V_{j}^{2}-1}{F_{2n+4}}+9U_{j}
^{2}F_{2n+4}\pm 6U_{j}V_{j},
\]
which together give
\begin{equation}\label{eq:Fk}
F_{k}=\pm 6U_{j}V_{j}+U_{j}^{2}(F_{2n+8}+9F_{2n+4}).
\end{equation}
We call this sequence
\begin{equation}\label{eq:Cj}
C_{j}^{\pm}:=\pm 6U_{j}V_{j}+U_{j}^{2}(F_{2n+8}+9F_{2n+4}),
\end{equation}
and aim to solve $F_{k}=C_{j}^{\pm}$ for positive integer $j$ and $k$. Note here that 
\begin{align*}
C_{1}^{-}&=-6U_{1}V_{1}+U_{1}^{2}(F_{2n+8}+9F_{2n+4})=-6F_{2n+6}+F_{2n+8}+9F_{2n+4}\\
&=F_{2n+5}-4F_{2n+6}+9F_{2n+4}=2F_{2n+4}-3F_{2n+3}=F_{2n}
\end{align*}
and that 
\begin{align*}
C_{1}^{+}&=6F_{2n+6}+F_{2n+8}+9F_{2n+4}=12F_{2n+6}+F_{2n}\\
&=F_{2n+10}+F_{2n+7}+2F_{2n+5}+6F_{2n+4}+F_{2n}\\
&=F_{2n+11}+F_{2n+5}+F_{2n+3}+2F_{2n}
\end{align*}
and therefore $F_{2n+11}<C_{1}^{+}<F_{2n+12}$. Hence we assume $j\geq 2$. In addition, because $X_{1}^{+}>X_{1}^{-}>0$ and
\begin{align*}
X_{j+1}^{\pm}&=3F_{2n+6}(3V_{j}\pm F_{2n+8}U_{j})+F_{2n+8}(\pm V_{j}+9F_{2n+4}U_{j})\\
&=(9F_{2n+6}\pm F_{2n+8})V_{j}+3F_{2n+8}(3F_{2n+4}\pm F_{2n+6})U_{j}\\
&>3V_{j}\pm F_{2n+8}U_{j}=X_{j}^{\pm},
\end{align*}
and since we're looking for solutions such that $F_{k}=\dfrac{(X_{j})^{2}-9}{F_{2n+8}}$, it may be assumed that $k>2n$ when $j\geq 2$.

\vspace*{5mm}

Define $\beta_{n}:=F_{2n+6}+\sqrt{F_{2n+6}^{2}-1}$. Then
\[
V_{j}=\frac{\beta_{n}^{j}+\beta_{n}^{-j}}{2},\text{ and } U_{j}=\frac{\beta_{n}^{j}-\beta_{n}^{-j}}{2\sqrt{F_{2n+4}F_{2n+8}}}.
\]
So $C_{j}^{\pm}$ can be rewritten as
\[
C_{j}^{\pm}=\pm 6\frac{\beta_{n}^{2j}-\beta_{n}^{-2j}}{4\sqrt{F_{2n+6}^{2}-1}}+(F_{2n+8}+9F_{2n+4})\cdot\frac{\beta_{n}^{2j}+
\beta_{n}^{-2j}-2}{4(F_{2n+6}^{2}-1)}
\]
\begin{gather*}
=\frac{\pm 6\beta_{n}^{2j}}{4\sqrt{F_{2n+6}^{2}-1}}+\frac{(F_{2n+8}+9F_{2n+4})\beta_{n}^{2j}}{4(F_{2n+6}^{2}-1)}-\frac{\pm 
6\beta_{n}^{-2j}}{4\sqrt{F_{2n+6}^{2}-1}}\\
+\frac{(F_{2n+8}+9F_{2n+4})\beta_{n}^{-2j}}{4(F_{2n+6}^{2}-1)}-\frac{(F_{2n+8}+9F_{2n+4})}{2(F_{2n+6}^{2}-1)}.
\end{gather*}
Define the sequence $\gamma_{n}^{\pm}$ by
\[
\gamma_{n}^{\pm}:=\frac{\pm 6}{4\sqrt{F_{2n+6}^{2}-1}}+\frac{(F_{2n+8}+9F_{2n+4})}{4(F_{2n+4}F_{2n+8})}.
\]
We have
\[
C_{j}^{\pm}=\beta_{n}^{2j}\gamma_{n}^{\pm}-\frac{(F_{2n+8}+9F_{2n+4})}{2(F_{2n+6}^{2}-1)}+\beta_{n}^{-2j}\gamma_{n}
^{\mp},
\]
and the problem can be expressed as finding solutions $j\geq 2$ and $k>2n$ to the equation
\begin{equation}\label{eq:Main Equation}
\beta_{n}^{2j}\gamma_{n}^{\pm}-\frac{(F_{2n+8}+9F_{2n+4})}{2(F_{2n+6}^{2}-1)}+\beta_{n}^{-2j}\gamma_{n}^{\mp}
=\frac{\alpha^{k}-\overline{\alpha}^{k}}{\sqrt{5}}.
\end{equation}



\section{A Linear Form in Three Logarithms}\label{D(9)3Logarithms}

We begin by finding bounds for $\gamma_{n}^{\pm}$.
\begin{lemma}\label{gamma bounds} $\gamma_{n}^{\pm}$ satisfy the following:
\begin{align*}
0.011\alpha^{-2n-4}&<\gamma_{n}^{-}<0.013\alpha^{-2n-4}\\
2.574\alpha^{-2n-4}&<\gamma_{n}^{+}<2.585\alpha^{-2n-4}.
\end{align*}
\end{lemma}

\proof
We have
\begin{align*}
\sqrt{\gamma_{n}^{\pm}}&=\frac{3}{2\sqrt{F_{2n+8}}}\pm\frac{1}{2\sqrt{F_{2n+4}}}\\
&=\frac{3}{2\sqrt{(\alpha^{2n+8}-\alpha^{-2n-8})/\sqrt{5}}}\pm\frac{1}{2\sqrt{(\alpha^{2n+4}-\alpha^{-2n-4})/\sqrt{5}}}\\
&=\frac{5^{1/4}\alpha^{-n-2}}{2}\Bigg(\frac{3}{\alpha^{2}\sqrt{(1-1/\alpha^{4n+16}})}\pm\frac{1}{\sqrt{(1-1/\alpha^{2n+8}})}\Bigg)
\end{align*}
As a result of the Taylor series of $(1-x)^{-1/2}$, we have for $0<x<1$
\[
1+\frac{1}{2}x<\frac{1}{\sqrt{1-x}}=1+\frac{1}{2}x+\frac{3}{8}x^{2}+\dots <1+\frac{x}{2}\Big(\frac{1}{1-x}\Big)
\]
so
\[
\frac{3}{\alpha^2}\Big(1+\frac{1}{2}\alpha^{-4n-16}\Big)<\frac{3}{\alpha^{2}\sqrt{(1-1/\alpha^{4n+16}})}<\frac{3}{\alpha^2}\Big(1+
\frac{\alpha^{-4n-16}}{2(1-\alpha^{-4n-16})}\Big),
\]
hence
\[
1.1458980<\frac{3}{\alpha^{2}}<\frac{3}{\alpha^{2}\sqrt{(1-1/\alpha^{4n+16}})}<1.14599721
\]
and similarly
\[
1<\frac{1}{\sqrt{(1-1/\alpha^{2n+8}})}<1.0041.
\]
We then obtain bounds for $\sqrt{\gamma_{n}^{\pm}}$
\[
0.141798<\frac{2}{5^{1/4}\alpha^{-n-2}}\sqrt{\gamma_{n}^{-}}<0.1499721
\]
and
\[
2.145898<\frac{2}{5^{1/4}\alpha^{-n-2}}\sqrt{\gamma_{n}^{+}}<2.15001,
\]
and finally the following bounds on $\gamma_{n}^{\pm}$:
\begin{align*}
0.011\alpha^{-2n-4}&<\gamma_{n}^{-}<0.013\alpha^{-2n-4}\\
2.574\alpha^{-2n-4}&<\gamma_{n}^{+}<2.585\alpha^{-2n-4}.
\end{align*}
\qed
%

Define a linear form in three logarithms, $\Lambda$ by
\[
\Lambda:=2j\log{\beta_{n}}-k\log{\alpha}+\log{(\sqrt{5}\gamma_{n}^{\pm})}.
\]
\begin{lemma}\label{4.2}
$0<\Lambda<1162\beta_{n}^{-2j}$ for $j\geq 2$.
\end{lemma}
\proof
From \eqref{eq:Main Equation},
\[
\beta_{n}^{2j}\gamma_{n}^{\pm}-\frac{\alpha^{k}}{\sqrt{5}}=\frac{(F_{2n+8}+9F_{2n+4})}{2(F_{2n+6}^{2}-1)}-
\frac{\overline{\alpha}^{k}}{\sqrt{5}}-\beta_{n}^{-2j}\gamma_{n}^{\mp}
\]
\noindent $\Lambda=\log{\sqrt{5}\gamma_{n}^{\pm}\beta_{n}^{2j}\alpha^{-k}}>0$ if and only if $\sqrt{5}\gamma_{n}^{\pm}\beta_{n}^{2j}\alpha^{-k}>1$. Therefore in order to show that $\Lambda>0$ we will assume, for a contradiction, that $\beta_{n}^{2j}\gamma_{n}^{+}\leq\dfrac{\alpha^{k}}{\sqrt{5}}$. This would mean that
\[
\frac{\sqrt{5}}{\alpha^{k}}\leq\frac{\beta_{n}^{-2j}}{\gamma_{n}^{\pm}}\leq\frac{\beta_{n}^{-2j}}{\gamma_{n}^{-}},
\]
and because $\dfrac{1}{F_{2n+4}}<\dfrac{9}{F_{2n+8}}$, and by \eqref{eq:Main Equation} and our assumption, this gives the following inequality
\begin{gather*}
\frac{1}{F_{2n+4}}<\frac{1}{2F_{2n+4}}+\frac{9}{2F_{2n+8}}=\frac{F_{2n+8}+9F_{2n+4}}{2(F_{2n+6}^2-1)}\\
\leq \beta_{n}^{-2j}\gamma_{n}^{\pm}+\frac{\overline{\alpha}^{k}}{\sqrt{5}}<\beta_{n}^{-2j}\Big(\gamma_{n}^{+}+\frac{1}
{5\gamma_{n}^{-}}\Big),
\end{gather*}
which we apply below along with the bounds for $\gamma_{n}^{\pm}$ obtained in \ref{gamma bounds}.
\begin{align*}
F_{2n+4}^{j}F_{2n+8}^{j}&=(F_{2n+6}^{2}-1)^{j}<\Big(2F_{2n+6}^{2}-1+2F_{2n+6}\sqrt{F_{2n+6}^{2}-1}\Big)^{j}=\beta_{n}^{2j}\\
&<F_{2n+4}\Big(\gamma_{n}^{+}+\frac{1}{5\gamma_{n}^{-}}\Big)<F_{2n+4}(2.585\alpha^{-2n-4}+18.2\alpha^{2n+4})
\end{align*}
and $F_{2n+4}^{j-1}F_{2n+8}^{j}<2.585\alpha^{-2n-4}+18.182\alpha^{2n+4}$ is only true if $j<2$, which is a contradiction. Hence $\beta_{n}^{2j}\gamma_{n}^{+}>\dfrac{\alpha^{k}}{\sqrt{5}}$ and so $\Lambda>0$.
Now we have
\begin{align*}
\Big|\alpha^{k}5^{-1/2}\beta_{n}^{-2j}(\gamma_{n}^{\pm})^{-1}-1\Big|&<\frac{1}{\beta_{n}^{2j}\gamma_{n}^{\pm}}
\Big(\frac{F_{2n+8}+9F_{2n+4}}{2(F_{2n+6}^{2}-1)}+\frac{1}{\sqrt{5}\alpha^{k}} \Big)\\
&<\frac{1}{\beta_{n}^{2j}\gamma_{n}^{\pm}}
\Big(\frac{2}{F_{2n+4}}+\frac{1}{\sqrt{5}\alpha^{2n+1}} \Big)\\
&<\frac{581}{\beta_{n}^{2j}}<\frac{1}{2},
\end{align*}
so by the fact that 
\begin{equation}\label{eq:e inequality}
|e^{\Lambda}-1|<\frac{1}{2}\text{ implies that } |\Lambda |<2|e^{\Lambda}-1|
\end{equation}
we must have that $\Lambda <1162\beta_{n}^{-2j}$.\qed
\vspace*{5mm}

Using these two lemmas we prove the following proposition
\begin{proposition}\label{4.3}
If equation \eqref{eq:Fk} has a positive integer solution $(j,k)$ with $j>1$ then
\[
j<1.15\cdot 10^{12}(2n+7)\log{(78j(2n+7))}.
\]
\end{proposition}
%
%
%
%
%
%
%
%

In order to prove this proposition, we will apply Lemma \ref{Matveev} to the linear form in three logarithms $\Lambda$ as defined above.
\[
\Lambda:=2j\log{\beta_{n}}-k\log{\alpha}+\log{(\sqrt{5}\gamma_{n}^{\pm})}
\]
We take
\begin{gather*}
N=3,\quad D=4,\quad b_{1}=2j,\quad b_{2}=-k,\quad b_{3}=1,\\
\alpha_{1}=\beta_{n},\quad \alpha_{2}=\alpha,\quad \alpha_{3}=\sqrt{5}\gamma_{n}^{\pm}.
\end{gather*}
Lemma \ref{Matveev} calls for $\alpha_{1},\alpha_{2},\alpha_{3}$ to be multiplicatively independent (i.e. that their logarithms be linearly independent). We have that $\alpha_{2}\in\mathbb{Q}(\sqrt{5})$ and $\alpha_{1},\alpha_{3}^{2}\in\mathbb{Q}(\sqrt{F_{2n+4}F_{2n+8}})$, and by Lemma \ref{first nosquare}, $F_{2n+4}F_{2n+8}$ is neither a square nor $5$ times a square. Therefore, writing $F_{2n+4}F_{2n+8}=du^{2}$ for $u,d\in\mathbb{Z}$, $5\neq d$ square-free, since no non-zero power of $\alpha_{2}$ can be in $\mathbb{Q}(\sqrt{d})$, if $\alpha_{1},\alpha_{2},\alpha_{3}$ are multiplicatively dependent we must have that $\alpha_{1}$ and $\alpha_{3}^{2}$ are multiplicatively dependent. Since $\alpha_{1}$ is a unit in $\mathbb{Q}(\sqrt{d})$, we must have that $\alpha_{3}^{2}=5(\gamma_{n}^{\pm})^{2}$ is a unit, but the norm of $5(\gamma_{n}^{\pm})^{2}$ is
\[
25(\gamma_{n}^{+}\gamma_{n}^{-})^{2}=25\Big(\frac{9F_{2n+4}-F_{2n+8}}{4F_{2n+4}F_{2n+8}}\Big)^{4}<1,
\]
so this is never an integer for any $n$, therefore $\alpha_{3}^{2}$ is not a unit for any $n$.

The absolute logarithmic heights for $\alpha_{1}$ and $\alpha_{2}$ are
\[
h(\alpha_{1})=h(\beta_{n})=\frac{1}{2}\log{\beta_{n}},\quad h(\alpha_{2})=h(\alpha)=\frac{1}{2}\log{\alpha}.
\]
For $\alpha_{3}$, since
\[
(x-\gamma_{n}^{+})(x-\gamma_{n}^{-})=x^{2}-2\Big(\frac{F_{2n+8}+9F_{2n+4}}{4(F_{2n+4}F_{2n+8})}\Big)x+
\Big(\frac{9F_{2n+4}-F_{2n+8}}{4F_{2n+4}F_{2n+8}}\Big)^{2},
\]
clearing denominators we get the minimal polynomial
\[
16F_{2n+4}^{2}F_{2n+8}^{2}x^{2}-8(F_{2n+8}^{2}F_{2n+4}+9F_{2n+4}^{2}F_{2n+8})x+(9F_{2n+4}-F_{2n+8})^{2},
\]
and since $|\gamma_{n}^{\pm}|\leq |\gamma_{n}^{+}|< 2.585\alpha^{-6}<1$, and $F_{\lambda}<\alpha^{\lambda}/\sqrt{5}$ for positive even $\lambda$ we have
\[
h(\gamma_{n}^{\pm})=\frac{1}{2}\log{(16F_{2n+4}^{2}F_{2n+8}^{2})}=\log{(4F_{2n+4}F_{2n+8})}<(4n+12)\log{\alpha}+\log{4/5}.
\]
Hence we can take
\begin{align*}
h(\alpha_{3})&=h(\sqrt{5}\gamma_{n}^{\pm})\leq h(\sqrt{5})+h(\gamma_{n}^{\pm})<\frac{1}{2}\log{5}+(4n+12)\log{\alpha}+
\log{4/5}\\
&<2\log{\alpha}+(4n+12)\log{\alpha}=(4n+14)\log{\alpha}
\end{align*}
and finally since we need $A_{i}\geq D\cdot h(\alpha_{i})$, we take
\[
A_{1}=2\log{\beta_{n}},\quad A_{2}=2\log{\alpha},\quad A_{3}=8(2n+7)\log{\alpha}.
\]
Next, since $\alpha^{\lambda-2}\leq F_{\lambda}\leq\alpha^{\lambda-1}$, it can be seen that
\[
\beta_{n}<2F_{2n+8}<2\alpha^{2n+5}<\alpha^{2n+7}
\]
and in addition,
\begin{align*}
\alpha^{k-1}&<2\alpha^{k-2}<2F_{k}\leq 12U_{j}V_{j}+2U_{j}^{2}(F_{2n+8}+9F_{2n+4})\\
&<20U_{j}V_{j}+F_{2n+4}F_{2n+8}U_{j}^{2}<(V_{j}+U_{j}\sqrt{F_{2n+4}F_{2n+8}})^{2}-V_{j}^{2}\\
&<(V_{j}+U_{j}\sqrt{F_{2n+4}F_{2n+8}})^{2}=(F_{2n+6}+\sqrt{F_{2n+4}F_{2n+8}})^{2j}\\
&<(2F_{2n+6})^{2j}<(2\alpha^{2n+5})^{2j}<\alpha^{2j(2n+7)}.
\end{align*}
Due to the results above, take
\begin{align*}
E&=\max{\{ 1,\max{\{|b_{j}|A_{j}/A_{N};1\leq j\leq N\}}\}}\\
&\leq\max{\Bigg\{ 2j,k,1,\frac{2j\log{\beta_{n}}}{\log{\alpha}},4j\log{\beta_{n}},\frac{k\log{\alpha}}{\log{\beta_{n}}},2k\log{\alpha},
4(2n+7),\frac{4(2n+7)\log{\alpha}}{\log{\beta_{n}}}\Bigg\}}\\
&=\max{\Big\{k,\frac{2j\log{\beta_{n}}}{\log{\alpha}},4(2n+7)\Big\}}
\leq 2j(2n+7)
\end{align*}
and
\begin{align*}
C(3)&=\frac{8}{2}(5)(9)(16e)^{4}<6.45\times 10^{8}\\
C_{0}&=\log{e^{20.2}3^{5.5}(16)\log{(4e)}}<30\\
W_{0}&=\log{(1.5eE\cdot 4\log{4e})}<\log{(78j(2n+7))}\\
\Omega&=(2\log{\beta_{n}})(2\log{\alpha})(8(2n+7)\log{\alpha}).
\end{align*}
\textit{Proof of Proposition \ref{4.3}}.
Apply Lemma \ref{Matveev} and combine this with Lemma \ref{4.2} to see
\[
2j\log{\beta_{n}}-\log{1162}<-\log{|\Lambda}|<1.434\cdot 10^{11}(2n+7)(\log{\beta_{n}})(\log{(78j(2n+7))}).
\]
Hence
\[
j<1.15\cdot 10^{12}(2n+7)\log{(78j(2n+7))}
\]
as desired.
\qed
%


\section{Linear Form in Two Logarithms}\label{D(9)2Logarithms}

Using $j=1,k=2n$ in $\Lambda$, define the linear form in three logarithms, $\Lambda_{0}$, by
\[
\Lambda_{0}:=2\log{\beta_{n}}-2n\log{\alpha}+\log{(\sqrt{5}\gamma_{n}^{\pm})}.
\]
The following upper bound applies to $\Lambda_{0}$:
\begin{lemma}\label{5.1}
$|\Lambda_{0}|<9473\beta_{n}^{-2}$.
\end{lemma}
\proof
Assume for now that $n\geq 2$. After substituting the one known solution, $j=1,k=2n$, into equation \eqref{eq:Main Equation} from earlier, it becomes
\[
\beta_{n}^{2}\gamma_{n}^{\pm}-\frac{\alpha^{2n}}{\sqrt{5}}=\frac{(F_{2n+8}+9F_{2n+4})}{2(F_{2n+6}^{2}-1)}-
\frac{\overline{\alpha}^{2n}}{\sqrt{5}}-\beta_{n}^{-2}\gamma_{n}^{\mp}.
\]
If $\beta_{n}^{2}\gamma_{n}^{\pm}\leq\dfrac{\alpha^{2n}}{\sqrt{5}}$, then $\dfrac{\alpha^{-2n}}{\sqrt{5}}\leq\dfrac{1}{5\beta_{n}^{2}\gamma_{n}^{\pm}}$ and
\begin{align*}
\Big| \alpha^{2n}5^{-1/2}\beta_{n}^{-2}(\gamma_{n}^{\pm})^{-1}-1\Big|&<\frac{\beta_{n}^{-2}\gamma_{n}^{\mp}+
\frac{\alpha^{-2n}}{\sqrt{5}}}{\beta_{n}^{2}\gamma_{n}^{\pm}}\\
&<\frac{\gamma_{n}^{\mp}+\frac{1}{5\gamma_{n}^{\pm}}}{\beta_{n}^{4}\gamma_{n}^{\pm}}\\
&<\frac{235.3+1656.2\alpha^{4n+8}}
{\beta_{n}^{4}}\\
&<\beta_{n}^{-2}(235.3/(55+12\sqrt{21})^2+1656.2)\\
&<1657\beta_{n}^{-2}<\frac{1}{2}.
\end{align*}
If $\beta_{n}^{2}\gamma_{n}^{\pm}>\dfrac{\alpha^{2n}}{\sqrt{5}}$, then
\begin{align*}
\Big| \alpha^{2n}5^{-1/2}\beta_{n}^{-2}(\gamma_{n}^{\pm})^{-1}-1\Big|&<\frac{1/(2F_{2n+4})+1/F_{2n+4}}{\beta_{n}
^{2}\gamma_{n}^{\pm}}\\
&<\frac{3}{2F_{2n+4}\beta_{n}^{2}\gamma_{n}^{\pm}}<306\beta_{n}^{-2}<\frac{1}{2}.
\end{align*}
For $n=1$, $|\Lambda_{0}|=2\log{(21+\sqrt{440})}-2\log{\alpha}+\log{(\sqrt{5}\gamma_{1}^{+})}<9473\beta_{n}^{-2}$. In every case we have (by inequality \eqref{eq:e inequality}) $|\Lambda_{0}|<9473\beta_{n}^{-2}$.\qed


\vspace{5mm}

The following linear form in two logarithms will help to give a hard bound for $j$. Define $\Lambda_{1}$ by
\[
\Lambda_{1}:=K\log{\alpha}-(j-1)\log{(5/4)},
\]
where $K=(2j-1)(2n+6)-k-6$. The following bound applies to $\Lambda_{1}$:
\begin{lemma}\label{5.2}
$|\Lambda_{1}|<(9j+17042)\alpha^{-4n-12}$.
\end{lemma}
\proof
Note first that
\begin{align*}
\beta_{n}&=F_{2n+6}+\sqrt{F_{2n+6}^{2}-1}=\frac{(F_{2n+6}+\sqrt{F_{2n+6}^{2}-1})^{2}}{F_{2n+6}+\sqrt{F_{2n+6}^{2}-1}}\\
&=\frac{2F_{2n+6}^{2}+2F_{2n+6}\sqrt{F_{2n+6}^{2}-1}-1}{F_{2n+6}+\sqrt{F_{2n+6}^{2}-1}}=2F_{2n+6}-\frac{1}{F_{2n+6}+
\sqrt{F_{2n+6}^{2}-1}}\\
&=2F_{2n+6}\Bigg(1-\frac{1}{2F_{2n+6}(F_{2n+6}+\sqrt{F_{2n+6}^{2}-1})}\Bigg)\\
\end{align*}
and $2F_{2n+6}=\dfrac{2}{\sqrt{5}}\Big(\alpha^{2n+6}-\overline{\alpha}^{2n+6}\Big)=\dfrac{2}{\sqrt{5}}\alpha^{2n+6}\Big(1-\dfrac{1}{\alpha^{4n+12}}\Big)$. So if we define $\delta_{n}$ by
\[
\delta_{n}=\Bigg(1-\frac{1}{2F_{2n+6}(F_{2n+6}+\sqrt{F_{2n+6}^{2}-1})}\Bigg)\Bigg(1-\frac{1}{\alpha^{4n+12}}\Bigg),
\]
then $\log{\beta_{n}}=\log{(\frac{2}{\sqrt{5}})}+(2n+6)\log{\alpha}+\log{\delta_{n}}$ and it can be seen that
\begin{align*}
\Lambda-\Lambda_{0}&=\Big(2j\log{\beta_{n}}-k\log{\alpha}+\log{(\sqrt{5}\gamma_{n}^{\pm})}\Big)-\Big(2\log{\beta_{n}}-2n\log{\alpha}+\log{(\sqrt{5}\gamma_{n}^{\pm})}\Big)\\
&=(2j-2)\log{\beta_{n}}-(k-2n)\log{\alpha}\\
&=(2j-2)\log{(2/\sqrt{5})}+K\log{\alpha}+(2j-2)\log{\delta_{n}},
\end{align*}
where $K=(2j-1)(2n+6)-k-6$. Hence we have $\Lambda_{1}=\Lambda-\Lambda_{0}-(2j-2)\log{\delta_{n}}$. We can bound $|\log{\delta_{n}}|$ by taking
\begin{align*}
|\log{\delta_{n}}|&\leq \Bigg|\log{\bigg(1-\frac{1}{2F_{2n+6}(F_{2n+6}+\sqrt{F_{2n+6}^{2}-1})}\bigg)}\Bigg|+\Bigg|\log{\bigg(1-\frac{1}{\alpha^{4n+12}}\bigg)}\Bigg|\\
&<\frac{1}{F_{2n+6}\Big(F_{2n+6}+\sqrt{F_{2n+6}^{2}-1}\Big)}+\frac{1}{\alpha^{4n+12}}<\frac{1}{2\alpha^{4n+8}}+\frac{1}{\alpha^{4n+12}}\\
&<\frac{9}{2\alpha^{4n+12}}.
\end{align*}
Here we have used the triangle inequality and \eqref{eq:e inequality}. This then gives
\[
|\Lambda_{1}|\leq |\Lambda|+|\Lambda_{0}|+|2j-2||\log{\delta_{n}}|<\frac{9473}{\beta_{n}^{2}}+\frac{9(j-1)}{\alpha^{4n+12}},
\]
and since
\begin{align*}
\beta_{n}&=F_{2n+6}+\sqrt{F_{2n+6}^{2}-1}>2\alpha^{2n+4}\text{, we must have}\\
\beta_{n}^{2}&>\frac{4}{7}\alpha^{4n+12}.
\end{align*}
Therefore
\[
|\Lambda_{1}|<(9j+17042)\alpha^{-4n-12}.
\]
\qed


%
\begin{lemma}\label{5.3}
If equation \eqref{eq:Fk} has a positive integer solution $(j,k)$ with $j>1$, then
\[
j<3.55\times 10^{19},\text{ and}\quad n<246806.
\]
\end{lemma}
\proof
Apply Lemma \ref{Laurent}. Let
\[
D=2,\quad \gamma_{1}=\frac{5}{4},\quad \gamma_{2}=\alpha,\quad b_{1}=(j-1),\quad b_{2}=K.
\]
In addition, take $h_{1}=\log{5},h_{2}=\dfrac{1}{2}$. By the Lemma \ref{5.2},
\begin{align*}
K&<\frac{(j-1)\log{(5/4)}+(9j+17042)\alpha^{-4n-12}}{\log{\alpha}}\\
&<0.4638(j-1)+0.0085j+16.0466<0.48j+15.59,
\end{align*}
and because
\[
\frac{|b_{1}|}{Dh_{2}}+\frac{|b_{2}|}{Dh_{1}}=(j-1)+\frac{K}{2\log{5}}<1.15j+3.85,
\]
Lemma \ref{Laurent} gives us a lower bound on $\Lambda_{1}$,
\[
\log{|\Lambda_{1}|}>-17.9\cdot 8\cdot\log{5}\cdot\Big(\max{\big\{\log{(1.15j+3.85)}+0.38,15\big\}}\Big)^{2}.
\]
From Lemma \ref{5.2}, it is seen that
\[
\log{|\Lambda_{1}|}<-(4n+12)\log{\alpha}+\log{(9j+17042)}.
\]
A combination of these two bounds yields
\[
n<119.75\Big(\max{\big\{\log{(1.15j+3.85)}+0.38,15\big\}}\Big)^{2}+0.52\log{(9j+17042)}.
\]
If
\[
\log{(1.15j+3.85)}+0.38\leq15
\]
then $j<1943527$ and $n<26952$. Otherwise,
\[
n<119.75\Big(\log{(1.15j+3.85)}+0.38\Big)^{2}+0.52\log{(9j+17042)}
\]
and we substitute this bound into Proposition \ref{4.3} to obtain
\begin{align*}
j<&1.15\cdot 10^{12}(2(119.75\Big(\log{(1.15j+3.85)}+0.38\Big)^{2}+0.52\log{(9j+17042)})+7)\\
&\times\log{(78j(2(119.75\Big(\log{(1.15j+3.85)}+0.38\Big)^{2}+0.52\log{(9j+17042)})+7))},\\
\end{align*}
which means that $j<3.55\times 10^{19}$ and so $n<246806$.
\qed



\section{Refining Our Bounds}\label{RefiningD(9)}

In this section, the bounds on $n$ and $j$ are improved before Baker-Davenport reduction is applied to those bounds in the next section. Lemma \ref{5.2} gives
\[
|K\log{\alpha}-(j-1)\log{(5/4)}|<(9j+17042)\alpha^{-4n-12}.
\]
Thus we can divide by $j-1$ to get
\[
\Big|\frac{\log{(5/4)}}{\log{\alpha}}-\frac{K}{j-1}\Big|<\frac{9j+17042}{(j-1)\alpha^{4n+12}\log{\alpha}}.
\]
Assume that
\begin{equation}\label{eq:6.2}
\frac{9j+17042}{(j-1)\alpha^{4n+12}\log{\alpha}}<\frac{1}{2(j-1)^{2}}.
\end{equation}
By above,
\[
\Big|\frac{\log{(5/4)}}{\log{\alpha}}-\frac{K}{j-1}\Big|<\frac{1}{2(j-1)^{2}}.
\]
By Theorem \ref{Legendre}, $\dfrac{K}{(j-1)}$ must be a convergent in the simple continued fraction expansion of $\log{(5/4)}/\log{\alpha}$. Since the denominator of the 46th convergent 
\[
\frac{25158053660121411107}{54253653513327093513}
\]
is greater than the upper bound of $3.55\times 10^{19}$ established for $j$, we can use the denominator of the 45th convergent
\[
\frac{4460457560349832575}{9619031832089360168}
\]
which is bigger than $9.6\times 10^{18}$ to obtain the following lower bound,
\begin{align*}
\Big|\frac{\log{(5/4)}}{\log{\alpha}}-\frac{K}{j-1}\Big|&\geq\Big|\frac{\log{(5/4)}}{\log{\alpha}}-\frac{4460457560349832575}{9619031832089360168}\Big| >1.9\times 10^{-39}.
\end{align*}
Combining these bounds
\[
1.9\times 10^{-39}<\frac{9j+17042}{(j-1)\alpha^{4n+12}\log{\alpha}}<17060\alpha^{-4n-12}(\log{\alpha})^{-1}
\]
gives $n<49$. It is also known from Lemma \ref{ParQuo} that if $\dfrac{p_{r}}{q_{r}}$ is the $r$th convergent of $\dfrac{\log{(5/4)}}{\log{\alpha}}$, then
\[
\bigg|\frac{\log{(5/4)}}{\log{\alpha}}-\frac{p_{r}}{q_{r}}\bigg|>\frac{1}{(a_{r+1}+2)q_{r}^{2}},
\]
where $a_{r+1}$ is the $(r+1)$st partial quotient of $\dfrac{\log{(5/4)}}{\log{\alpha}}$. Therefore for $2\leq r\leq 45$,
\[
\min{\bigg\{\frac{1}{(a_{r+1}+2)(j-1)^{2}}\bigg\}}<\frac{9j+17042}{(j-1)\alpha^{4n+12}\log{\alpha}}.
\]
Since $\max{\{a_{r+1}:2\leq r\leq 45\}}=a_{36}=49$,
\[
\alpha^{4n+12}<51(j-1)(9j+17042)(\log{\alpha})^{-1}.
\]

If \eqref{eq:6.2} does not hold, i.e., if
\[
\frac{9j+17042}{(j-1)\alpha^{4n+12}\log{\alpha}}\geq\frac{1}{2(j-1)^{2}},
\]
then
\[
\alpha^{4n+12}\leq 2(j-1)(9j+17042)(\log{\alpha})^{-1}.
\]
In either case, 
\[
\alpha^{4n+12}<51(j-1)(9j+17042)(\log{\alpha})^{-1}<444352j^{2}.
\]
This leads to the following proposition.
\begin{proposition}
If equation \eqref{eq:Fk} has a positive integer solution $(j,k)$ with $j>1$, then
\[
n<1.04\log{j}+3.76
\]
\end{proposition}

When this bound is combined with the bound for $j$ in Proposition \ref{4.3}, we get
\[
j<1.15\cdot 10^{12}(2.08\log{j}+14.52)\log{(78j(2.08\log{j}+14.52))},
\]
which implies
\begin{lemma}
If equation \eqref{eq:Fk} has a positive integer solution $(j,k)$ with $j>1$, then $j<4.63\times 10^{15}$ and $n<42$.
\end{lemma}
%



\section{Reduction of the Bounds}\label{ReductionD(9)}

We will now use the reduction method of Baker and Davenport to bring the bounds for $j$ and $n$ down to something more computationally manageable. We will then use a procedure written in $\text{Maple}^{\text{TM}}$ to check all remaining possibilities.

We know that
\[
0<2j\log{\beta_{n}}-k\log{\alpha}+\log{(\sqrt{5}\gamma_{n}^{\pm})}<1162\beta_{n}^{-2j}.
\]
In order to apply Baker-Davenport reduction, consider
\[
\kappa=\frac{2\log{\beta_{n}}}{\log{\alpha}},\quad \mu=\frac{\log{(\sqrt{5}\cdot\gamma_{n}^{\pm})}}{\log{\alpha}},\quad A=\frac{1162}{\log{\alpha}},\quad B=\beta_{n}^{2},\quad M=4.63\times 10^{15}.
\]
We then used a set of procedures written in Maple to undertake the computations. In all cases, we obtained $j\leq 6$ and therefore $1\leq n\leq 6$. So the following result holds.
\begin{lemma}
If equation \eqref{eq:Fk} has a positive integer solution $(j,k)$ with $j>1$, then $j\leq 6$ and $n\leq 6$.
\end{lemma}

Applying this result to equation \eqref{eq:Cj} in order to prove Theorem 1.1, it is seen that no combination of $n$ and $j$ with $1\leq n\leq$, $2\leq j\leq6$ yields a Fibonacci number. We have already established that $F_{11}<C_{1}^{+}<F_{12}$, which means the only solution is $C_{1}^{-}=F_{2n}$. When $n=1$, $F_{2n}=F_{2}=1=F_{1}$, giving one extra solution in that case.


\section{Setting up the $D(64)$-Triple}\label{SettingupD(64)}
We now show that for the $D(64)$-triple $\{F_{k}, F_{2n+12}, 16F_{2n+6}\}$, the only value $k$ can take is $2n$.  Note that in this case, $a=16F_{2n+6} $ and $b=F_{2n+12}$ in Lemma \ref{Filipin}, and $64\mid a$ if we assume that $3\mid n$,  which means that all solutions of the equation $4Y\sqrt{F_{2n+6}}+X\sqrt{F_{2n+12}}$ are given by
\[
4Y\sqrt{F_{2n+6}}+X\sqrt{F_{2n+12}}=(\pm 32\sqrt{F_{2n+6}} +8\sqrt{F_{2n+12}})\bigg(\frac{F_{2n+9}+\sqrt{F_{2n+6}F_{2n+12}}}{2}\bigg)^{j},
\] 
where $j\geq 0$.

  Again define sequences $ (U_{j})_{j\geq 0}  \text{ and }  (V_{j})_{j\geq 0}, \text{ this time by}$
  \[
 V_{j}+U_{j} \sqrt{F_{2n+6}F_{2n+12}}:=(F_{2n+9}+\sqrt{F_{2n+6}F_{2n+12}})^{j}.
  \]
   This gives
   \[
    y=y_{j}=\pm 8V_{j}+2U_{j}F_{2n+12} \text{ and } x=x_{j}=8V_{j}\pm 32F_{2n+6}U_{j},
    \]
   so
    \begin{equation}\label{eq:Fk64}
   F_{k}= \pm32 U_{j}V_{j}+4U_{j}^{2}(F_{2n+12}+16F_{2n+6}).
   \end{equation}
   If we let $C_{j}^{(\pm)}:=\pm32 U_{j}V_{j}+4U_{j}^{2}(F_{2n+12}+16F_{2n+6})$ for $j\geq 1$ then our goal now is to solve 
   \begin{equation}\label{eq:Cj64}
   F_{k}= C_{j}^{(\pm)}
   \end{equation}
   for positive $j$ and $k$.
   
   As in the previous section, the equation has the solution $C_{j}^{(-)}=F_{2n}$, so to prove the result it must be shown that there are no other solutions.  We can note immediately that
   \[
   F_{2n+14}<C_{1}^{(+)}<F_{2n+15},
   \]
   so suppose for a contradiction that $j\geq 2$.  In this case,
   \[
   \beta_{n}:=\frac{F_{2n+9}+\sqrt{F_{2n+6}F_{2n+12}}}{2},
   \]
    so we can express
   \[
    V_{j}:=\frac{\beta_{n}^{j}+\beta_{n}^{-j}}{2} \text{ and } U_{{j}}:=\frac{\beta_{n}^{j}-\beta_{n}^{-j}}{2\sqrt{F_{2n+9}^{2}-4}}.
   \]
   We find that $C_{j}^{(\pm)}=\beta_{n}^{2j}\gamma_{n}^{(\pm)}-2\Big(\dfrac{F_{2n+12}+16F_{2n+6}}{F_{2n+9}^{2}-4}\Big)+\beta_{n}^{-2j}\gamma_{n}^{(\mp)}$ where 
   \[
   \gamma_{n}^{(\pm)}:=\pm \frac{8}{\sqrt{F_{2n+9}^{2}-4}}+\frac{F_{2n+12}+16F_{2n+6}}{F_{2n+9}^{2}-4}. 
   \]
    So we must solve:
   \[
    C_{j}^{(\pm)}=\beta_{n}^{2j}\gamma_{n}^{(\pm)}-2\bigg(\frac{F_{2n+12}+16F_{2n+6}}{F_{2n+9}^{2}-4}\bigg)+\beta_{n}^{-2j}\gamma_{n}^{(\mp)}=\frac{\alpha^{k}-\overline{\alpha}^{k}}{\sqrt5}    .
    \]
    

\section{A Linear Form in Three Logarithms (2)}\label{D(64)3Logarithms}
    \begin{lemma}
   The following bounds apply to $\gamma_{n}^{\pm}$,
\begin{align*}
0.00694321\alpha^{-2n-4}&<\gamma_{n}^{-}<0.00733800\alpha^{-2n-4}\\
8.45276900\alpha^{-2n-4}&<\gamma_{n}^{+}<8.46635831\alpha^{-2n-4}.
\end{align*} 
    \end{lemma}
    Using the same identity from earlier,
    \[
    \sqrt{\gamma_{n}^{(\pm)}}=\bigg(\frac{4}{\sqrt{F_{2n+12}}}\pm \frac{1}{\sqrt{F_{2n+6}}}\bigg)=5^{1/4}\alpha^{-n-3}\Bigg(\frac{4}{\alpha^3\sqrt{1-\frac{1}{\alpha^{4n+24}}}}\pm \frac{1}{\sqrt{1-\frac{1}{\alpha^{4n+12}}}}\Bigg), 
    \]
    giving
    \[ 
    1+\dfrac{1}{2}\bigg(\dfrac{1}{\alpha^{4n+24}}\bigg)<\dfrac{1}{\sqrt{1-\frac{1}{\alpha^{4n+24}}}}<1+\dfrac{\frac{1}{\alpha^{4n+24}}}{2(1-\frac{1}{\alpha^{4n+24}})}<\dfrac{4}{\alpha^{3}}\cdot \dfrac{1}{\sqrt{1-\frac{1}{\alpha^{4n+24}}}}, 
    \]
    so
    \[ 
    0.9442719<\frac{4}{\alpha^{3}}\frac{1}{\sqrt{1-\frac{1}{\alpha^{4n+24}}}}<0.9442765.
    \]
    Thus: 
    \[
    1<\dfrac{1}{\sqrt{1-\frac{1}{\alpha^{4n+24}}}}<1+\dfrac{\frac{1}{\alpha^{4n+12}}}{2(1-\frac{1}{\alpha^{4n+12}})}<1.00155765.
    \]
    Rearranging
    \[
     1.9442719<\dfrac{\sqrt{\gamma_{n}^{+}}}{5^{1/4}\alpha^{-n-3}}<1.94583415 
  \]
  and
  \[
  -0.0572858<\dfrac{\sqrt{\gamma_{n}^{-}}}{5^{1/4}\alpha^{-n-3}}<-0.0557234
  \]
   gives the result. \qed

  \vspace{5mm}  
    Define a linear form in logarithms, $\Lambda :=2j\log{(\beta_{n})}-k\log{(\alpha)}+\log{(\sqrt{5}(\gamma_{n}^{(\pm)})}$.
    
    \begin{lemma}
    $0<\Lambda<3868\beta_{n}^{-2j}$ for $j\geq 2$.
    \end{lemma}
    \proof
    We have 
    \[
    \beta_{n}^{2j}\gamma_{n}^{(\pm)}-\dfrac{\alpha^{k}}{\sqrt{5}}=\dfrac{2(F_{2n+12}+16F_{2n+6})}{F_{2n+9}^{2}-4}-\beta_{n}^{-2j}\gamma_{n}^{(\mp)}-\dfrac{\overline\alpha^{k}}{\sqrt{5}}.
    \]
    Suppose for a contradiction that $\beta_{n}^{2j}\gamma_{n}^{(\pm)}\leq \dfrac{\alpha_{n}^{k}}{\sqrt{5}}$, thus
    \[
    \dfrac{\sqrt{5}}{\alpha^{k}}\leq\dfrac{\beta_{n}^{-2j}}{\gamma_{n}^{(\pm)}}\leq \dfrac{\beta_{n}^{-2j}}{\gamma_{n}^{(-)}}, 
    \]
    so 
    \begin{align*}
    \dfrac{1}{F_{2n+6}}&<\dfrac{2}{F_{2n+6}}+\dfrac{32}{F_{2n+12}}=\dfrac{2(F_{2n+12}+16F_{2n+6})}{F_{2n+9}^{2}-4}\\
   &<\beta_{n}^{-2j}\gamma_{n}^{(\mp)}+\dfrac{\overline\alpha^{k}}{\sqrt{5}}<\beta_{n}^{-2j}\gamma_{n}^{(+)}+\dfrac{1}{\sqrt{5}\alpha^{k}}\\
   &\leq\beta_{n}^{-2j}\gamma_{n}^{(+)}+\dfrac{1}{\sqrt{5}}\bigg(\dfrac{\beta_{n}^{-2j}}{\sqrt{5}\gamma_{n}^{(-)}}\bigg)=\beta_{n}^{-2j}\bigg(\gamma_{n}^{(+)}+\dfrac{1}{5\gamma_{n}^{(-)}}\bigg).
   \end{align*}
    Since
    \[
     F_{2n+6}^{j}F_{2n+12}^{j}<\beta_{n}^{2j}<F_{2n+6}\bigg(\gamma_{n}^{(+)}+\frac{1}{5\gamma_{n}^{(-)}}\bigg)<F_{2n+6}(8.47\alpha^{-2n-6}+28.81\alpha^{2n+6}),
     \]
    it is seen that
    \[
    F_{2n+6}^{j-1}F_{2n+12}^{j}<8.47\alpha^{-2n-6}+28.81\alpha^{2n+6},
    \]
    which is a contradiction.  Therefore $\beta_{n}^{2j}\gamma_{n}^{(\pm)}-\dfrac{\alpha^{k}}{\sqrt{5}}>0$, so $\Lambda>0$.
    
    Furthermore, 
    \begin{align*}
    |\alpha^{k}5^{-1/2}\beta_{n}^{-2j}(\gamma_{n}^{(\pm)})^{-1}-1|&=\frac{1}{\beta_{n}^{2j}\gamma_{n}^{(\pm)}}\bigg|\frac{\alpha^{k}}{\sqrt{5}}-\beta_{n}^{2j}\gamma_{n}^{(\pm)}\bigg|=\frac{1}{\beta_{n}^{2j}\gamma_{n}^{(\pm)}}\bigg(\beta_{n}^{2j}\gamma_{n}^{(\pm)}-\frac{\alpha^{k}}{\sqrt{5}}\bigg)\\
    &=\frac{1}{\beta_{n}^{2j}\gamma_{n}^{(\pm)}}\bigg[\frac{2(F_{2n+12}+16F_{2n+6})}{F_{2n+9}^{2}-4}-\beta_{n}^{-2j}\gamma_{n}^{(\mp)}-\frac{\overline\alpha^{k}}{\sqrt{5}}\bigg]\\
    &<\frac{1}{\beta_{n}^{2j}\gamma_{n}^{(\pm)}}\bigg[\frac{2(F_{2n+12}+16F_{2n+6})}{F_{2n+9}^{2}-4}+\frac{1}{\alpha^{k}\sqrt{5}}\bigg]\\
        &\leq\frac{1}{\beta_{n}^{2j}\gamma_{n}^{(\pm)}}\bigg[\frac{2(F_{2n+12}+16F_{2n+6})}{F_{2n+9}^{2}-4}+\frac{1}{\alpha^{2n+1}\sqrt{5}}\bigg].
        \end{align*}
        We assume $k$ is odd (for $k$ even the inequality is clear), and $k\geq2n$ from earlier.  We then have that the above is equal to:
        \begin{align*}
        &=\frac{1}{\beta_{n}^{2j}\gamma_{n}^{(\pm)}}\bigg[\frac{2}{F_{2n+6}}+\frac{32}{F_{2n+12}}+\frac{1}{\alpha^{2n+1}\sqrt{5}}\bigg]\\
        &<\frac{144.026}{\beta_{n}^{2j}}\alpha^{2n+6}\bigg[\frac{2\sqrt{5}}{\alpha^{2n+6}-\frac{1}{\alpha^{2n+6}}}+\frac{32\sqrt{5}}{\alpha^{2n+12}-\frac{1}{\alpha^{2n+12}}}+\frac{1}{\alpha^{2n+1}\sqrt{5}}\bigg]\\
        &=\frac{144.026}{\beta_{n}^{2j}}\bigg[\frac{2\sqrt{5}}{1-\frac{1}{\alpha^{4n+12}}}+\frac{32\sqrt{5}}{\alpha^{6}-\frac{1}{\alpha^{4n+18}}}+\frac{\alpha^{5}}{\sqrt{5}}\bigg]\\
        &\leq\frac{144.026}{\beta_{n}^{2j}}\bigg[\frac{2\sqrt{5}}{1-\frac{1}{\alpha^{16}}}+\frac{32\sqrt{5}}{\alpha^{6}-\frac{1}{\alpha^{22}}}+\frac{\alpha^{5}}{\sqrt{5}}\bigg]<\frac{1934}{\beta_{n}^{2j}}<\frac{1}{2}.
        \end{align*}
    So $|\Lambda |<\dfrac{3868}{\beta_{n}^{2j}}$\qed
    
    \begin{proposition}\label{prp:64}
   If  \eqref{eq:Fk64} has a positive integer solution $(j,k)$ with $j>1$, then 
   \[
   j<1.16\cdot10^{10}(2n+9)\log{(156j(n+5))}.
   \]
    \end{proposition}
    \proof
    Take \begin{gather*}
    N=3,\quad D=4,\quad b_{1}=2j,\quad b_{2}=-k,\quad b_{3}=1,\\
     \alpha_{1}=\beta_{n},\quad \alpha_{2}=\alpha,\quad \alpha_{3}=\sqrt{5}\gamma_{n}^{(\pm)}.
    \end{gather*}  
    From Lemma \ref{second nosquare}, $\alpha_{1}, \alpha_{2}, \alpha_{3}$ are multiplicatively independent.  
    
    The logarithmic heights are
    \begin{gather*}
    h( \alpha_{1})=h(\beta_{n})=\frac{1}{2}\log{\beta_{n}}, \text{ and } h(\alpha_{2})=h(\alpha)=\frac{1}{2}\log{\alpha}.
    \end{gather*}  
    We now calculate $h(\alpha_{3}).$
    \begin{gather*}
    (X-\gamma_{n}^{(+)})(X-\gamma_{n}^{(-)})=X^{2}-(\gamma_{n}^{(+)}+\gamma_{n}^{(-)})X+\bigg(\frac{F_{2n+12}-16F_{2n+6}}{F_{2n+6}F_{2n+12}}\bigg)^{2}\\
    =F_{2n+6}^{2}F_{2n+12}^{2}X^{2}-2(F_{2n+12}^{2}F_{2n+6}+16F_{2n+12}F_{2n+6}^{2})X+(F_{2n+12}-16F_{2n+6})^{2},
    \end{gather*}
  so
  \begin{align*}
  h(\gamma_{n}^{(\pm)})&=\frac{1}{2}\bigg[\log{(F_{2n+6}^{2}F_{2n+12}^{2})}+\log{(1)}+\log{(1)}\bigg]=\log{(F_{2n+6}F_{2n+12})}\\
  &<\log{\bigg( \frac{\alpha^{2n+6}\alpha^{2n+12}}{5} \bigg)}=(4n+18)\log{(\alpha)}+\log{\bigg(\frac{1}{5}\bigg)}.
  \end{align*}
  Thus
  \begin{align*}
  h(\alpha_{3})&=h(\sqrt{5} \gamma_{n}^{(\pm)})\leq h(\sqrt{5})+h(\gamma_{n}^{(\pm)})\\
  &<\frac{1}{2}\log{(5)}+(4n+18)\log{(\alpha)}+\log{\bigg(\frac{1}{5}\bigg)}\\
  &=(4n+18)\log{(\alpha)}.
  \end{align*}
  We have
  \begin{align*} 
  A_{1}&=\max{\{2\log{(\beta_{n})},|\log{(\beta_{n})|}\}}=2\log({\beta_{n}}),\\
   A_{2}&=\max{\{2\log{\alpha}, |\log{\alpha|}\}}=2\log{\alpha,}\\
    A_{3}&=\max{\{4(4n+18)\log{(\alpha)},|\log{(\sqrt{5}\gamma_{n}^{(\pm)})}|\}}
  =8(2n+9)\log{(\alpha)}.
  \end{align*}
  As $\alpha^{l-2}\leq F_{l}\leq\alpha^{l-1}$, it can be seen that $\beta_{n}<2F_{2n+9}<2\alpha^{2n+8}<\alpha^{2(n+5)}$.  
  
  \noindent Moreover, 
  
  \begin{align*}
  \alpha^{k-1}&<2\alpha^{k-2}<2F_{k}\leq64U_{j}V_{j}+8U_{j}^{2}(F_{2n+12}+16F_{2n+6})\\
  &<(V_{j}+U_{j}\sqrt{F_{2n+6}F_{2n+12}})^{2}=\bigg(\frac{F_{2n+9}+\sqrt{F_{2n+9}^{2}-4}}{2}\bigg)^{2j}\\
&<F_{2n+9}^{2j}<(\alpha^{2n+8})^{2j}<\alpha^{4j(n+4)}.
\end{align*}
Let
\begin{align*}
E&\leq 4j(n+5) \\
C(3)&<6.45\times10^{8} \\
C_{0}&<30\\
W_{0}&=\log{(1.5eE4\log{(4e)})}<\log{(156j(n+5))}\\
\Omega&=A_{1}A_{2}A_{3}\leq 32(2n+9)\log{(\alpha)}\log{(\beta_{n})}.
\end{align*}
This gives by Lemma \ref{Matveev}
\begin{align*}
\log{|\Lambda|}&>-C(N)C_{0}W_{0}D^{2}\Omega\\
&\geq-6.45\times10^{8}\cdot30\log{(156j(n+5))}(32(2n+9)(\log{(\alpha)})^{2}(\log{(\beta_{n})})).
\end{align*}
So
\begin{align*}
\log{\Lambda}&<\log{\frac{3868}{\beta_{n}^{2j}}}=\log{(3868)}-2j\log{(\beta_{n})}\\
j&<1.16\times10^{10}(2n+9)\log{(156j(n+5))}.
\end{align*}
\qed
  
    
    \section{Linear Form in Two Logarithms (2)}\label{D(64)2Logarithms}
\begin{lemma}
Define $\Lambda_{0}=2\log{\beta_{n}}-2n\log{\alpha}+\log{(\sqrt{5}\gamma_{n}^{(\pm)})}$.  The following bound holds
\[
 |\Lambda_{0}|<\dfrac{270830}{\beta_{n}^{2}}.
\]
\end{lemma}

\proof
Assume that $n\geq 2$.  
We have
\[
\beta_{n}^{2}\gamma_{n}^{(\pm)}-\frac{\alpha^{2n}}{\sqrt{5}}=\frac{2(F_{2n+12}+16F_{2n+6})}{F_{2n+9}^{2}-4}-\beta_{n}^{-2}\gamma_{n}^{(\mp)}-\frac{\alpha^{-2n}}{\sqrt{5}}.
\]
If $\dfrac{1}{\beta_{n}^{2}\gamma_{n}^{(\pm)}}\geq\alpha^{-2n}\sqrt{5}$, then $\dfrac{1}{5\beta_{n}^{2}\gamma_{n}^{(\pm)}}\geq\dfrac{\alpha^{-2n}}{\sqrt{5}}$, and
\begin{gather*}
|\alpha^{2n}5^{-\frac{1}{2}}\beta_{n}^{-2}(\gamma_{n}^{(\pm)})^{-1}-1|=\frac{1}{\beta_{n}^{2}\gamma_{n}^{(\pm)}}\bigg| \frac{\alpha^{2n}}{\sqrt{5}}-\beta_{n}^{2}\gamma_{n}^{(\pm)}\bigg|\\
<\frac{\beta_{n}^{-2}\gamma_{n}^{(\mp)}+\dfrac{\alpha^{-2n}}{\sqrt{5}}}{\beta_{n}^{2}\gamma_{n}^{(\pm)}}<\frac{         \gamma_{n}^{(\mp)}+\dfrac{1}{5\gamma_{n}^{(\pm)}}     }                 {\beta_{n}^{4}\gamma_{n}^{(\pm)}                }\\
<\frac{8.46635831+\dfrac{\alpha^{4n+12}}{5(0.00694321)}}{0.00694321\beta_{n}^{4}}\\
<\frac{1219.4+4148.7\alpha^{4n+12}}{\beta_{n}^{4}}\\
\leq \frac{ \dfrac{1219.4}{\beta_{n}^{2}}+\dfrac{4148.7\alpha^{4n+12}}{\beta_{n}^{2}}}{\beta_{n}^{2}}<4149\beta_{n}^{-2}.
\end{gather*}
If $\beta_{n}^{2}\gamma_{n}^{(\pm)}>\dfrac{\alpha^{2n}}{\sqrt{5}}$, then
\begin{gather*}
|\alpha^{2n}5^{-\frac{1}{2}}\beta_{n}^{-2}(\gamma_{n}^{(\pm)})^{-1}-1|=\frac{1}{\beta_{n}^{2}\gamma_{n}^{(\pm)}}(\beta_{n}^{2}\gamma_{n}^{(\pm)}-\alpha^{2n}5^{-\frac{1}{2}})\\
<\frac{1}{\beta_{n}^{2}\gamma_{n}^{(\pm)}}\frac{2(F_{2n+12}+16F_{2n+6})}{F_{2n+9}^{2}-4}=\frac{\bigg( \dfrac{2}{F_{2n+6}}+\dfrac{32}{F_{2n+12}}\bigg)}{\beta_{n}^{2}\gamma_{n}^{(\pm)}}\\
<\frac{\bigg( \dfrac{2}{F_{2n+6}}+\dfrac{32}{F_{2n+12}}\bigg)}{\beta_{n}^{2}(0.00694321)\alpha^{-2n-6}}=\frac{\dfrac{2\alpha^{2n+6}}{F_{2n+6}}+\dfrac{32\alpha^{2n+6}}{F_{2n+12}}}{\beta_{n}^{2}(0.00694321)}\\
<\frac{2\alpha^{2}+\dfrac{32}{\alpha^{4}}}{\beta_{n}^{2}(0.00694321)}<\frac{9.91}{\beta_{n}^{2}(0.00694321)}<\frac{1428}{\beta_{n}^{2}}.
\end{gather*}
For $n=1$, $\Lambda_{0}\leq 2\log{\beta_{1}}-2\log{\alpha}+\log{(\sqrt{5}\gamma_{1}^{(+)})} <\dfrac{270830}{\beta_{n}^{2}}$.  In all cases we have 
\[
|\Lambda_{0}|<\dfrac{270830}{\beta_{n}^{2}}.
\]

\qed

\begin{lemma}
Let 
\[
\Lambda_{1}:=K\log{\alpha}-(j-1)\log{\bigg(\frac{5}{4}\bigg)}, \text{ where } K:=(2j-1)(2n+9)-k-9,  
\]
then
\[ 
|\Lambda_{1}|<\dfrac{10(j+5493)}{\alpha^{4n+14}}.
\]
\end{lemma}
\proof
We know 
\begin{gather*}
\beta_{n}=F_{2n+9}+\sqrt{F_{2n+9}^{2}-4}=2F_{2n+9}-\frac{4}{F_{2n+9}+\sqrt{F_{2n+9}^{2}-4}}\\
=2F_{2n+9}\bigg(1-\frac{4}{2F_{2n+9}(F_{2n+9}+\sqrt{F_{2n+9}^{2}-4})}\bigg)
\end{gather*}
and 
\[
2F_{2n+9}=\frac{2}{\sqrt{5}}(\alpha^{2n+9}-\overline{\alpha}^{2n+9})=\frac{2}{\sqrt{5}}\alpha^{2n+9}\bigg(1+\frac{1}{\alpha^{4n+18}}\bigg).
\]
Define 
\[
\delta_{n}:=\bigg(1-\frac{4}{2F_{2n+9}(F_{2n+9}+\sqrt{F_{2n+9}^{2}-4})}\bigg)\bigg(1+\frac{1}{\alpha^{4n+18}}\bigg).
\]
Then $\log{(\beta_{n})}=\log{\bigg(\dfrac{2}{\sqrt{5}}\bigg)}+(2n+9)\log{\alpha}+\log{\delta_{n}}$.
\begin{align*}
\Lambda-\Lambda_{0}=&(2j-2)\log{(\beta_{n})}-(k-2n)\log{(\alpha)}\\
=&(2j-2)\log{\bigg(\frac{2}{\sqrt{5}}\bigg)}+(2j-2)(2n+9)\log{(\alpha)}\\
&+(2j-2)\log{(\delta_{n})}-(k-2n)\log{(\alpha)}\\
=&(2j-2)\log{(\delta_{n})}+K\log{(\alpha)}-(j-1)\log{\bigg(\frac{5}{4}\bigg)}.
\end{align*}
So $\Lambda_{1}=\Lambda-\Lambda_{0}-(2j-2)\log{(\delta_{n})}$.  By the triangle inequality,
\[
|\log{(\delta_{n})}|\leq \bigg| \log{\bigg( 1-\frac{4}{2F_{2n+9}(F_{2n+9}+\sqrt{F_{2n+9}^{2}-4})}\bigg)} \bigg|+\bigg| \log{ \bigg(1+\frac{1}{\alpha^{4n+18}}\bigg)}\bigg|,
\]
so
\[
|\log{(\delta_{n})}|<\frac{4}{F_{2n+9}(F_{2n+9}+\sqrt{F_{2n+9}^{2}-4})}+\frac{1}{\alpha^{4n+18}}
\]
\[
<\frac{4}{\alpha^{4n+14}}+\frac{1}{\alpha^{4n+18}}<\frac{5}{\alpha^{4n+14}},
\]
and
\[
|\Lambda_{1}|\leq |\Lambda|+|\Lambda_{0}|+|2j-2||\log{\delta_{n}}|<\frac{3868}{\beta_{n}^{2}}+\frac{270830}{\beta_{n}^{2}}+\frac{10(j-1)}{\alpha^{4n+14}}.
\]
Clearly 
\[
\beta_{n}=F_{2n+9}\bigg(1+\sqrt{1-\frac{4}{F_{2n+9}^{2}}}\bigg)\geq F_{2n+9}\bigg(1+\sqrt{\frac{89^{2}-4}{89^{2}}}\bigg)>\frac{\alpha^{2n+9}}{\sqrt{5}}\bigg(1+\sqrt{\frac{7917}{89^{2}}}\bigg).
\]
Thus
 \[
 \beta_{n}^{2}>\alpha^{4n+18}\bigg(1+\frac{\sqrt{7917}}{89}\bigg)^{2}>\frac{3}{4}\alpha^{4n+18}>5\alpha^{4n+14},
\]
and we have $|\Lambda_{1}|<\dfrac{274698}{\beta_{n}^{2}}+\dfrac{10(j-1)}{\alpha^{4n+14}}<\dfrac{10(j+5493)}{\alpha^{4n+14}}$.
\qed
\vspace{5mm}


\section{Refining Our Bounds (2)}\label{RefiningD(64)}

\begin{lemma}
If \eqref{eq:Fk64} has a positive integer solution $(j,k)$ with $j>1$ then
\[
j<6.25\times10^{21}\text{ and } n<37024766.
\]
\end{lemma}
\proof
Apply Lemma \ref{Matveev}. Set
\begin{gather*}
D=2,\quad \gamma_{1}=\frac{5}{4},\quad \gamma_{2}=\alpha,\quad b_{1}=(j-1), \quad b_{2}=K.
\end{gather*}
  Take $h_{1}=\log{(5)}$, and $h_{2}=\dfrac{1}{2}$.  We have $K\log{(\alpha)}-(j-1)\log{(\dfrac{5}{4})}<\dfrac{10(j+5493)}{\alpha^{4n+14}}$, 
  so 
\[
K<\frac{(j-1)\log{\bigg(\dfrac{5}{4}\bigg)}+\dfrac{10(j+5493)}{\alpha^{4n+14}}}{\log{(\alpha)}}<19.3+0.47j,
\]
so take
\[
\frac{|b_{1}|}{Dh_{2}}+\frac{|b_{2}|}{Dh_{1}}=(j-1)+\frac{K}{2\log(5)}\leq 4.996+1.15j.
\]
Let $b'=5+1.15j$, hence
\[
\log{|\Lambda_{1}|}>(-17.9)\cdot16\cdot(\max\{\log{(5+1.15j)}+0.38,15\})^{2}\cdot \log{(5)}\cdot \frac{1}{2}.
\]
 The previous Lemma gives: $\log{|\Lambda_{1}|}<-(4n+14)\log{(\alpha)}+\log{(10(j+5493))}$, so we have 
\[
(-17.9)\cdot16\cdot(\max\{\log{(5+1.15j)}+0.38,15\})^{2}\cdot \log{(5)}\cdot \frac{1}{2}<\log{|\Lambda_{1}|}
\]
\[
<-(4n+14)\log{(\alpha)}+\log{(10(j+5493))}.
\]
Rearranging this,
\[ 
n<120(\max\{{log{(5+1.15j)}}+0.38,15\})^{2}+0.52\cdot \log{(10(j+5493))}.
\]
  If $\log{(1.15j+5)}+0.38<15$, then $j<1943957$ and $n<27009$.  Otherwise 
  \[
  n<120(\log{(5+1.15j)}+0.38)^{2}+0.52\log{(10(j+5493))}.
  \]
    Combining this with $j<1.16\times10^{10}(2n+9)\log{(156j(n+5))}$ gives $j<6.25\times10^{21}$ and $n<37024766$.
\qed
\vspace{5mm}
\begin{proposition}
If \eqref{eq:Fk64} has a positive integer solution $(j,k)$, with $j>1$, then 
\[
n<1.04\cdot \log{j}+4.6.
\]
\end{proposition}
\proof
We have that $|K\log{(\alpha)}-(j-1)\log{\bigg(\dfrac{5}{4}\bigg)|}<10(j+5493)\alpha^{-(4n+14)}$,  
hence
\[
 \bigg|\frac{\log{(5/4)}}{\log{(\alpha)}}-\frac{K}{(j-1)}\bigg|<\frac{10(j+5493)}{(j-1)\alpha^{4n+14}\log{(\alpha)}}.
\]
First, assume 
\begin{equation}\label{eq:last64}
\frac{10(j+5493)}{(j-1)\alpha^{4n+14}\log{(\alpha)}}<\frac{1}{2(j-1)^{2}}.
\end{equation}
  Then 
  \[
  \bigg|\frac{\log{(5/4)}}{\log{(\alpha)}}-\frac{K}{(j-1)}\bigg|<\frac{1}{2(j-1)^{2}}.
  \]
    The denominator of the $49$th convergent of $\dfrac{\log{(5/4)}}{\log{(\alpha)}}$ is greater than $6.25\times10^{21}$, our upper bound of $j$.  The $48$th convergent gives the lower bound:
\[
\bigg|\frac{\log{(5/4})}{\log{(\alpha)}}-\frac{K}{j-1}\bigg|>4\times10^{-44}.
\]
Combining the bounds
\[
4\times10^{-44}<\frac{10(j+5493)}{(j-1)\alpha^{4n+14}\log{(\alpha)}}<5513\cdot \alpha^{-4n-14}(\log{(\alpha))}^{-1}. 
\]
 gives us $n<55$.

Since we know that $\dfrac{p_{r}}{q_{r}}$ is the $r$th convergent, $\bigg|\dfrac{\log{(5/4)}}{\log{(\alpha)}}-\dfrac{p_{r}}{q_{r}}\bigg|>\dfrac{1}{(a_{r+1}+2)q_{r}^{2}}$, and since $\max{\{a_{r+1}:2\leq r\leq 47\}}=a_{36}=49$, 
\[
\frac{1}{51(j-1)^{2}}<\frac{10(j+5493)}{(j-1)\alpha^{4n+14}\log{(\alpha)}}\implies \alpha^{4n+14}<\frac{510(j-1)(j+5493)}{\log{(\alpha)}}.
\]
If \eqref{eq:last64} does not hold,

\[
\frac{1}{2(j-1)^{2}} \leq \frac{10(j+5493)}{(j-1)\alpha^{4n+14}\log{(\alpha)}}\implies \alpha^{4n+14}\leq \frac{20(j-1)(j+5493)}{\log{(\alpha)}}.
\]
In both cases
\[
\alpha^{4n+14}<\frac{510(j-1)(j+5493)}{\log{(\alpha)}}<1060j(j+5493)<5823640j^{2},
\]
so $n<1.04\cdot \log{(j)}+4.6$.
\qed

\vspace{5mm}
Combining the above proposition with \ref{prp:64} yields:
\vspace{5mm}
\begin{lemma}
If \eqref{eq:Fk64} has a positive integer solution $(j,k)$ with $j>1$, then
\[
 j<4.88\times 10^{15} \text{ and } n<43.
 \]
 \end{lemma}

\section{Baker-Davenport Reduction}\label{D(64)Reduction}
Using the same method as before, we have obtained bounds on $j$ and $n$.  All that remains to prove Theorem 1.2 is to apply Baker-Davenport reduction to these bounds.  

We have
\[
0<2j\log{\beta_{n}}-k\log{\alpha}+\log{(\sqrt{5}\cdot \gamma_{n}^{(\pm)})}<3868\beta_{n}^{-2j}.
\]
To apply Baker-Davenport reduction, consider
\begin{gather*}
\kappa=\frac{2\log{\beta_{n}}}{\log{\alpha}}, \quad \mu=\frac{\log{(\sqrt{5}\cdot \gamma_{n}^{(\pm)})}}{\log{\alpha}}, \quad A=\frac{3868}{\log{\alpha}}, \quad M=4.88\times 10^{15}
\end{gather*}
We again use procedures written in $\text{Maple}$  to find $j\leq 5$, which in turn gives us $1\leq n \leq 5$.
\begin{lemma}
If \eqref{eq:Fk64} has a positive integer solution $(j,k)$ with $j>1$, then
\[
 j\leq 5 \text{ and } n\leq 5.
 \]
\end{lemma}
All possible $j$ between 2 and 5 and $n$ from 1 to 5 can be very quickly checked in \eqref{eq:Cj64} to see that no solution exists.  It was already seen that $F_{14}<C_{1}^{(+)}<F_{15}$, so the only possible solution to \eqref{eq:Cj64} is $C_{1}^{(-)}=F_{2n}$.  When $n=1$, $F_{2n}=F_{2}=1=F_{1}$, so there is an additional solution in this case again.

\section*{acknowledgements}
The authors would like to thank Dr. Claude Levesque. His advice on this paper was invaluable. The authors thank the referee for a careful reading of the manuscript and for comments which improved this paper. The third author was supported in part by NSERC.

\end{document}